\documentclass[11pt,letterpaper]{article}
\usepackage{graphicx}
\usepackage{epsfig}
\usepackage{amssymb}
\usepackage{amsthm}
\usepackage{amsfonts}
\usepackage{amsmath}
\usepackage{multicol,multirow}
\usepackage[margin=1in]{geometry}

\usepackage[T1]{fontenc}
\usepackage[utf8]{inputenc}
\usepackage{authblk}

\usepackage{enumitem}
\usepackage{subcaption}
\usepackage{hhline}

\newtheorem{theorem}{Theorem}
\newtheorem{lemma}{Lemma}
\newtheorem{corollary}[lemma]{Corollary}
\newtheorem{proposition}[lemma]{Proposition}
\newtheorem{remark}[lemma]{Remark}
\newtheorem{example}[lemma]{Example}
\newtheorem{definition}[lemma]{Definition}

\newcommand{\plant}{\Gamma}
\newcommand{\obs}{\mathfrak{G}}
\newcommand{\Ap}{A_{\plant}}
\newcommand{\Ag}{A_{\obs}}
\newcommand{\Sym}{\mathbb{S}}
\newcommand{\tr}{\mathrm{trace}}
\newcommand{\sh}{\mathcal{SH}}
\newcommand{\R}{\mathbb{R}}
\newcommand{\eig}{\mathcal{E}}
\newcommand{\proj}{\mathcal{P}}
\newcommand{\egap}{\mathrm{eigengap}}
\newcommand{\norm}[1]{\left\lVert#1\right\rVert} 
\newcommand{\krank}[1]{\mathrm{kruskal}(#1)}
\DeclareMathOperator*{\argmin}{arg\,min}
\DeclareMathOperator*{\argmax}{arg\,max}

\title{Finding Planted Subgraphs with Few Eigenvalues using the Schur-Horn Relaxation \thanks{The authors were supported in part by NSF Career award CCF-1350590 and by Air Force Office of Scientific Research grant
FA9550-14-1-0098.}}

\author[1]{Utkan Onur Candogan} 
\author[2]{Venkat Chandrasekaran}
\affil[1]{Department of Electrical Engineering, California Institute of Technology, \authorcr Pasadena, CA, 91125. Email: utkan@caltech.edu}   
\affil[2]{Departments of Computing and Mathematical Sciences and of Electrical Engineering, California Institute of Technology, Pasadena, CA, 91125. Email: venkatc@caltech.edu }

\begin{document}

\maketitle

\begin{abstract}
Extracting structured subgraphs inside large graphs -- often known as the planted subgraph problem -- is a fundamental question that arises in a range of application domains.  This problem is NP-hard in general, and as a result, significant efforts have been directed towards the development of tractable procedures that succeed on specific families of problem instances.  We propose a new computationally efficient convex relaxation for solving the planted subgraph problem; our approach is based on tractable semidefinite descriptions of majorization inequalities on the spectrum of a symmetric matrix.  This procedure is effective at finding planted subgraphs that consist of few distinct eigenvalues, and it generalizes previous convex relaxation techniques for finding planted cliques.  Our analysis relies prominently on the notion of \emph{spectrally comonotone} matrices, which are pairs of symmetric matrices that can be transformed to diagonal matrices with sorted diagonal entries upon conjugation by the same orthogonal matrix.
\end{abstract}


\section{Introduction} \label{Intro}

In application domains ranging from computational biology to social data analysis, graphs are frequently used to model relationships among large numbers of interacting entities.  A commonly encountered question across many of these application domains is that of identifying structured subgraphs inside larger graphs.  For example, identifying specific motifs or substructures inside gene regulatory networks is useful in revealing higher-order biological function \cite{artymiuk1994graph,dobrin2004aggregation,mason2007graph}. Similarly, extracting completely connected subgraphs in social networks is useful for determining communities of people that are mutually linked to each other \cite{leskovec2010empirical,mishra2007clustering,radicchi2004defining}.  In this paper, we propose a new algorithm based on convex optimization for finding structured subgraphs inside large graphs, and we give conditions under which our approach succeeds in performing this task.

Formally, suppose $\plant$ and $\obs$ are graphs\footnote{Throughout this paper we consider undirected, unweighted, loopless graphs.} on $k$ nodes and $n$ nodes (here $n > k$), respectively, with the following property: there exists a subset of vertices $V \subset \{1,\dots,n\}$ with $|V| = k$ such that the induced subgraph of $\obs$ corresponding to the vertex set $V$ is isomorphic to $\plant$.    The \emph{planted subgraph} problem is to identify the vertex subset $V$ given the graphs $\obs$ and $\plant$; see Figure \ref{ClebschGraphTogether} for an example.  The decision version of the planted subgraph problem is known as the induced subgraph isomorphism problem in the theoretical computer science literature, and it has been shown to be NP-hard \cite{karp1972reducibility}.  Nevertheless, as this problem arises in a wide range of application domains as described above, significant efforts have been directed towards the development of computationally tractable procedures that succeed on certain families of problem instances.  Much of the focus of this attention has been on the special case of the \emph{planted clique} problem in which the subgraph $\plant$ is fully connected.  Alon et al. \cite{alon1998finding} and Feige and Krauthgamer \cite{feige2000finding} developed a spectral algorithm for the planted clique problem, and subsequently Ames and Vavasis \cite{ames2011nuclear} described an approach based on semidefinite programming with similar performance guarantees to the earlier work based on spectral algorithms. Conceptually, these methods are based on a basic observation about the spectrum of a clique, namely that the adjacency matrix of a clique on $k$ nodes has two distinct eigenvalues, one with multiplicity equal to one and the other with multiplicity equal to $k-1$.  We describe a new semidefinite programming technique that generalizes the method of Ames and Vavasis \cite{ames2011nuclear} to planted subgraphs $\plant$ that are not fully connected, with the spectral properties of $\plant$ playing a prominent role in our algorithm and our analysis.

\begin{figure}
\centering
   \subcaptionbox{\label{ClebschGraph}}{\includegraphics[scale=0.25]{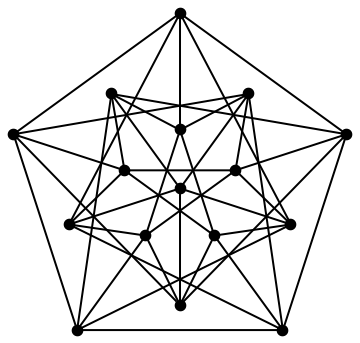}} \hspace{0.5in}%
   \subcaptionbox{\label{ClebschOnPlane}}{\includegraphics[scale=0.13]{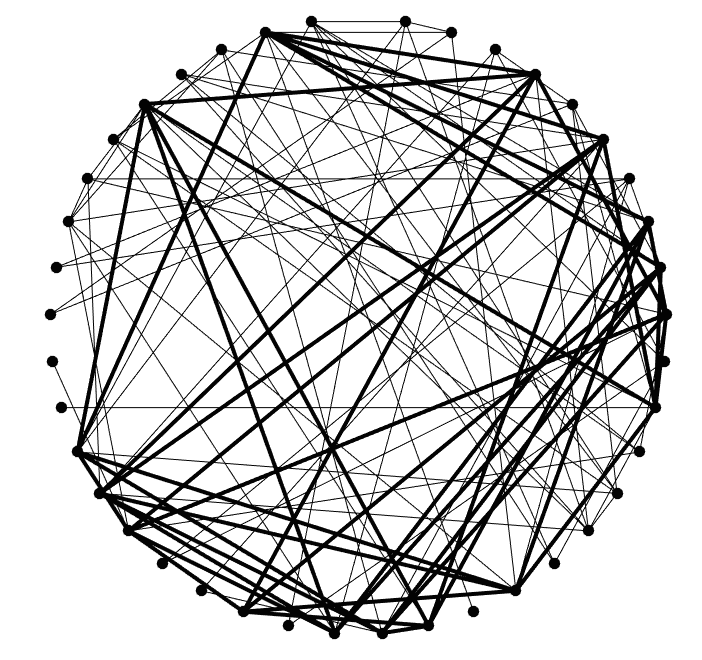}}
   \caption{The Clebsch graph ($16$ nodes) on the left.  An example on the right of a $40$-node graph containing the Clebsch graph as an induced subgraph; the thick edges correspond to a $16$-node induced subgraph that is isomorphic to the Clebsch graph.}\label{ClebschGraphTogether}
\end{figure}

\subsection{Our Contributions}

Let $\Ap \in \Sym^k$ and $\Ag \in \Sym^n$ represent the adjacency matrices of $\plant$ and of $\obs$, with $\Sym^q$ denoting the space of $q \times q$ real symmetric matrices. Given any matrix $M \in \Sym^k$, we let $[M]_{k \rightarrow n} \in \Sym^n$ for $n > k$ denote an $n \times n$ symmetric matrix with the leading principal minor of order $k$ equal to $M$ and all the other entries equal to zero.  The following combinatorial optimization problem is a natural first approach to phrase the planted subgraph problem in a variational manner:
\begin{equation}\label{Unsolvable Opt}
\begin{aligned}
\hat{A}_{co} = \argmax_{A \in \Sym^n} & ~ \tr(A \cdot A_\obs) \\ \mathrm{s.t.} & ~ A_{i,j} = 0 ~\mathrm{if}~ (A_\obs)_{i,j} = 0 ~\mathrm{and}~ i \neq j \\ & ~ A \in \{\Pi [A_\plant]_{k \rightarrow n} \Pi' ~|~ \Pi ~\mathrm{is~an}~ n \times n ~\mathrm{permutation~matrix} \}.
\end{aligned}
\end{equation}
Assuming that there is no other subgraph of $\obs$ that is isomorphic to $\plant$, one can check that the optimal solution $\hat{A}_{co}$ of this problem identifies the vertices $V \subset \{1,\dots,n\}$ whose induced subgraph in $\obs$ is isomorphic to $\plant$, i.e., the unique optimal solution $\hat{A}_{co}$ is equal to zero everywhere except for the principal minor corresponding to the indices in $V$ and $(\hat{A}_{co})_{V,V} = \tilde{\Pi} A_\plant \tilde{\Pi}'$ for some $k \times k$ permutation matrix $\tilde{\Pi}$.  However, solving ($\ref{Unsolvable Opt}$) is intractable in general.  Replacing the combinatorial constraint $A \in  \{\Pi [A_\plant]_{k \rightarrow n} \Pi' ~|~ \Pi ~\mathrm{is~an}~ n \times n ~\mathrm{permutation~matrix} \}$ with the convex constraint $A \in \mathrm{conv}\{\Pi [A_\plant]_{k \rightarrow n} \Pi' ~|~ \Pi ~\mathrm{is~an}~ n \times n ~\mathrm{permutation~matrix} \}$ does not lead to a tractable problem as checking membership in the polytope $\mathrm{conv}\{\Pi [A_\plant]_{k \rightarrow n} \Pi' ~|~ \Pi ~\mathrm{is~an}~ n \times n ~\mathrm{permutation~matrix} \}$ is intractable for general planted graphs $\plant$ (unless P $=$ NP).

We describe next a convex outer approximation of the set $\{\Pi [A_\plant]_{k \rightarrow n} \Pi' \allowbreak ~|~ \allowbreak \Pi ~\mathrm{is~an}~ \allowbreak n \times n \allowbreak ~\mathrm{permutation}~\allowbreak\mathrm{matrix} \}$ that leads to a tractable convex program.  For any matrix $M \in \Sym^n$, the \emph{Schur-Horn orbitope} $\sh(M) \subset \Sym^n$ is defined as \cite{sanyal2011orbitopes}:
\begin{equation} \label{SH convex hull}
\sh(M) = \mathrm{conv}\{U M U' ~|~ U ~\mathrm{is~an}~ n \times n ~\mathrm{orthogonal~matrix} \}.
\end{equation}
The term `orbitope' was coined by Sanyal, Sottile, and Sturmfels in their work on convex hulls of orbits generated by the actions of groups, and the Schur-Horn orbitope was so named by these authors due to its connection to the Schur-Horn theorem in linear algebra \cite{sanyal2011orbitopes}.  In combinatorial optimization, approximations based on replacing permutations matrices by orthogonal matrices have also been employed to obtain bounds on the Quadratic Assignment Problem \cite{FinBR1987}.  The set $\sh(M)$ depends only on the eigenvalues of $M$, and it is clearly an outer approximation of the set $\{\Pi M \Pi' ~|~ \Pi ~\mathrm{is~an}~ n \times n ~\mathrm{permutation~matrix} \}$.  Crucially for our purposes, the Schur-Horn orbitope $\sh(M)$ for any $M \in \Sym^n$ has a tractable semidefinite description via majorization inequalities on the spectrum of a symmetric matrix \cite{ben2001lectures,sanyal2011orbitopes}; see Section~\ref{subsec:sdpschurhorn}.  Hence, we propose the following tractable semidefinite programming relaxation for the planted subgraph problem:
\begin{equation}\tag{$P$}\label{PrimalOptimization}
\begin{aligned}
\hat{A}_{sh} = \argmax_{A \in \Sym^n} & ~ \tr(A \cdot A_\obs) \\ \mathrm{s.t.} & ~ A_{i,j} = 0 ~\mathrm{if}~ (A_\obs)_{i,j} = 0 ~\mathrm{and}~ i \neq j \\ & ~ A \in \sh([A_\plant - \gamma I_k]_{k \rightarrow n}).
\end{aligned}
\end{equation}
Here $I_k \in \Sym^k$ is the $k \times k$ identity matrix. We refer to this convex program as the \emph{Schur-Horn relaxation}, and this problem can be solved to a desired precision in polynomial time.  This relaxation only requires knowledge of the eigenvalues of the planted graph $\plant$. The parameter $\gamma \in \R$ is to be specified by the user, and we discuss suitable choices for $\gamma$ in the sequel.  Note that changing $A_\plant$ to $A_\plant - \gamma I_k$ in the constraints of \eqref{Unsolvable Opt} essentially leaves that problem unchanged (the nonzero principal minor of the optimal solution simply changes from $\hat{A}_{\mathrm{co}}$ to $\hat{A}_\mathrm{co} - \gamma I_k$).  However, the additional degree of freedom provided by the parameter $\gamma$ plays a more significant role in the Schur-Horn relaxation as it allows for shifts of the spectrum of $A_\plant$ to more favorable values, which is essential for the solution of various planted subgraph problems; see Section \ref{Subsection OptConds} for further details, as well as the experiments in Section \ref{SectionNumerical} for numerical illustrations.  We say that the Schur-Horn relaxation succeeds in recovering the planted subgraph $\plant$ if the optimal solution $\hat{A}_{sh} \in \Sym^n$ satisfies the following conditions: the optimal solution $\hat{A}_{sh}$ is unique, the submatrix $(\hat{A}_{sh})_{V,V} = \tilde{\Pi} A_\plant \tilde{\Pi}' - \gamma I_k$ for some $k \times k$ permutation matrix $\tilde{\Pi}$, and the remaining entries of $\hat{A}_{sh}$ are equal to zero.

In Section \ref{SectionPropertySH} we study the geometric properties of the Schur-Horn orbitope as these pertain to the optimality conditions of the Schur-Horn relaxation.  Our analysis relies prominently on the notion of \emph{spectrally comonotone} matrices, which refers to a pair of symmetric matrices that can be transformed to diagonal matrices with sorted diagonal entries upon conjugation by the same orthogonal matrix.  Spectral comonotonicity is a more restrictive condition than simultaneous diagonalizability, and it enables a precise characterization of the normal cones at extreme points of the Schur-Horn orbitope (Proposition \ref{Spect Comonotone Prop}).  This discussion leads directly to the central observation of our paper that the Schur-Horn relaxation is useful for finding planted graphs $\Gamma$ that consist of \emph{few distinct eigenvalues}.  Cliques form the simplest examples of such graphs as their spectrum consists of two distinct eigenvalues.  There are numerous other graph families whose spectrum consists of few distinct eigenvalues, and the study of such graphs is a significant topic in graph theory \cite{bridges1981multiplicative, doob1970graphs,doob1970characterizing,muzychuk1998graphs,van1995regular, van1996graphs,van1998small}. For example, strongly regular graphs are (an infinite family of) regular graphs with three distinct eigenvalues; the Clebsch graph of Figure \ref{ClebschGraphTogether} is a strongly regular graph on $16$ nodes with eigenvalues in the set $\{5, 1, -3\}$ and degree equal to five.  For a more extensive list of graphs with few eigenvalues, see Section \ref{SubsectionGraphsFewEvals}.

We state and prove the main theoretical result of this paper in Section \ref{Subsection MainResult} -- see Theorem \ref{Main Theorem}.  If the planted subgraph $\plant$ and its complement are both \emph{symmetric} -- $\plant$ and its complement are both vertex- and edge-transitive -- and if $\plant$ is connected, then this theorem takes on a simpler form (Corollary \ref{Main Corollary}).  Specifically, the success of the Schur-Horn relaxation \eqref{PrimalOptimization} relies on the existence of a suitable eigenspace $\eig \subset \R^k$ of $A_\plant$.  Concretely, let $\proj_\eig \in \Sym^{k}$ denote the projection onto $\eig$, and let $\mu(\eig) = \max_{i,j, ~ i\neq j} \frac{|(\proj_\eig)_{i,j}|}{\sqrt{|(\proj_\eig)_{i,i}| |(\proj_\eig)_{j,j}|}}$ denote the coherence of $\eig$.  Assuming that the edges in $\obs$ outside the induced subgraph $\Gamma$ are placed independently and uniformly at random with probability $p \in [0, \frac{1}{\mu(\eig) k})$ (i.e., the Erd\H{o}s-R\'{e}nyi random graph model), we show in Corollary $\ref{Main Corollary}$ that the Schur-Horn relaxation (\ref{PrimalOptimization}) with parameter\footnote{In our experiments in Section \ref{SectionNumerical}, we set $\gamma$ equal to the eigenvalue of $A_\plant$ with the largest multiplicity.  See Section \ref{SectionRecovering} for further discussion.} $\gamma = \lambda_\eig$ (the eigenvalue associated to $\eig$) succeeds with high probability provided:
\begin{equation*}
n \lesssim \min_{\substack{\lambda ~\mathrm{eigenvalue~of}~A_\plant \\ \lambda \neq \lambda_\eig}} \min\Bigg\{ |\lambda-\lambda_\eig|^2 ~ \frac{\dim(\eig)^2 \big( 1-kp\mu(\eig)\big)}{k^2\,p}, \left(|\lambda-\lambda_\eig|-2 |\lambda_\eig| \right)^2 \Bigg\}  + k.
\end{equation*}
The coherence parameter $\mu(\eig)$ lies in $(0,1]$, and it appears prominently in results on sparse signal recovery via convex optimization \cite{donoho2003optimally}.  In analogy to that literature, a small value of $\mu(\eig)$ is useful in our context (informally) to ensure that the planted graph $\plant$ looks sufficiently `different' from the remainder of $\obs$ (see Section \ref{SectionRecovering} for details).  Thus, the Schur-Horn relaxation succeeds if the planted graph $\plant$ consists of few distinct eigenvalues that are well-separated, and in which one of the eigenspaces has a small coherence parameter associated to it.  For more general non-symmetric graphs, our main result (Theorem \ref{Main Theorem}) is stated in terms of a parameter associated to an eigenspace $\eig$ of $A_\plant$ called the \emph{combinatorial width}, which roughly measures the average conditioning over all minors of $\proj_\eig$ of a certain size.

\paragraph{Specialization to the planted clique problem} The sum of the adjacency matrix of a clique and the identity matrix has rank equal to one, and consequently the planted clique problem may be phrased as one of identifying a rank-one submatrix inside a larger matrix (up to shifts of the diagonal by the identity matrix).  In her thesis \cite{fazel2002matrix}, Fazel proposed the nuclear norm as a tractable convex surrogate for identifying low-rank matrices in convex sets, and subsequent efforts provided theoretical support for the effectiveness of this relaxation in a range of rank minimization problems \cite{candes2009exact,recht2010guaranteed}.  Building on these ideas, Ames and Vavasis \cite{ames2011nuclear} proposed a nuclear norm minimization approach for the planted clique problem.  The Schur-Horn relaxation (\ref{PrimalOptimization}) specializes to the relaxation in \cite{ames2011nuclear} when $\plant$ is the clique.  Specifically, letting $A_\textrm{clique} \in \Sym^k$ denote the adjacency matrix of a $k$-clique, one can check that:
\begin{equation} \label{SHClique}
\sh([A_\textrm{clique} + I_k]_{k \rightarrow n}) = \{P \in \mathbb{S}^n ~|~ \mathrm{trace}(P) = k, ~ P \succeq 0\}.
\end{equation}
As the nuclear norm of a positive semidefinite matrix is equal to its trace, the Schur-Horn orbitope $\sh([A_\plant + I_k]_{k \rightarrow n})$ is simply a face of the nuclear norm ball in $\Sym^n$ scaled by a factor $k$.  Thus, the Schur-Horn relaxation (\ref{PrimalOptimization}) with $\gamma=-1$ is effectively a nuclear norm relaxation when the planted subgraph of interest is the clique.\footnote{The nuclear norm relaxation in \cite{ames2011nuclear} is formulated in a slightly different fashion compared to the Schur-Horn relaxation (\ref{PrimalOptimization}) for the case of the planted clique; specifically, one can show that our relaxation succeeds whenever the nuclear norm relaxation in \cite{ames2011nuclear} succeeds.}  Further, our main result (Theorem \ref{Main Theorem}) can be specialized to the case of a planted clique to obtain the main result in \cite{ames2011nuclear}; see Corollary \ref{Corollary Clique}.

\subsection{Paper Outline} \label{Outline}
In Section \ref{SectionPropertySH} we discuss the geometric properties of the Schur-Horn orbitope and their connection to the optimality conditions of the Schur-Horn relaxation, along with an extensive list of families of graphs with few eigenvalues.  Section \ref{SectionRecovering} contains our main theoretical results, while in Section \ref{SectionNumerical} we demonstrate the utility of the Schur-Horn relaxation in practice via numerical experiments.  We conclude in Section \ref{SectionDiscussion} with a discussion of further research directions.

\paragraph{Notation} The \emph{normal cone} at a point $x \in \mathcal{C}$ for a closed, convex set $\mathcal{C} \subset \R^n$ is denoted by $\mathcal{N}_\mathcal{C}(x)$ and it is the collection of linear functionals that attain their maximal value over $\mathcal{C}$ at $x$ \cite{rockafellar2015convex}.  The projection operator onto a subspace $\eig \subset \R^n$ is denoted by $\proj_\eig$.  The restriction of a linear map $A: \R^n \to \R^n$ to an invariant subspace $\eig$ of $A$ is denoted by $A|_{\eig}: \eig \to \eig$. The orthogonal complement of a subspace $\eig$ is denoted by $\eig^\perp$. The notation $\mathrm{dim}(\eig)$ denotes the dimension of a subspace $\eig$. The \emph{eigengap} of a symmetric matrix $M \in \mathbb{S}^n$ associated to an invariant subspace $\eig \subset \R^n$  of $M$ is defined as:
\begin{equation*}
\begin{aligned}
\egap(M, \eig) = \min\big\{|\lambda_\eig - \lambda_{\eig^\perp}| ~\big|~ & \lambda_\eig \text{ an eigenvalue of } M|_\eig, \\ & \lambda_{\eig^\perp} \text{ an eigenvalue of } M|_{\eig^\perp} \big\}.
\end{aligned}
\end{equation*}
The smallest and largest eigenvalues of a symmetric matrix $A$ are represented by $\lambda_{\min}(A)$ and $\lambda_{\max}(A)$, respectively.  The norms $\norm{\cdot}, \norm{\cdot}_2$, and $\norm{\cdot}_F$ denote the vector $\ell_2$ norm, the matrix operator/spectral norm, and the matrix Frobenius norm, respectively.  The vector $1_\ell \in \R^\ell$ denotes the all-ones vector of length $\ell$.  We denote the identity matrix of size $k$ by $I_k$.  The matrix $I_\Omega \in \R^{|\Omega| \,\times\, k}$ denotes the matrix whose rows are the rows of $I_k$ indexed by $\Omega \subset \{1,\dots,\,k\}$, so that the rows of $I_{\Omega} A$ are the rows of $A$ indexed by $\Omega$ for any $A\in \R^{k \times q}$.  The matrix $A_{\Omega,\Omega} \in \R^{|\Omega|\times |\Omega|}$ denotes the principal minor of $A$ indexed by the set $\Omega$.  The group of $n \times n$ orthogonal matrices is denoted by $\mathcal{O}_n \subset \R^{n\times n}$.  The set $\mathrm{relint}(\mathcal{C})$ specifies the relative interior of any set $\mathcal{C}$.  The column space of a matrix $A$ is denoted by $\mathrm{col}(A)$.  The quantity $\mathbb{E}[\cdot]$ denotes the usual expected value, where the distribution is clear from context.

\section{Geometric Properties of the Schur-Horn Orbitope} \label{SectionPropertySH}

In this section, we analyze the optimality conditions of the Schur-Horn relaxation from a geometric perspective. In particular, the notion of a pair of \emph{spectrally comonotone matrices} plays a central role in our development, and we elaborate on this point in the next subsection.  Based on this discussion, we observe that the Schur-Horn relaxation is especially useful for finding planted graphs consisting of few distinct eigenvalues, and we give examples of graphs with this property in Section \ref{SubsectionGraphsFewEvals}.  The main theoretical results formalizing the utility of the Schur-Horn relaxation are presented in Section \ref{SectionRecovering}.

\subsection{Optimality Conditions of the Schur-Horn Relaxation} \label{Subsection OptConds}
We state the optimality conditions of the Schur-Horn relaxation in terms of the normal cones at extreme points of the Schur-Horn orbitope:

\begin{lemma} \label{First Optimality Lemma}
Consider a planted subgraph problem instance in which the nodes of $\obs$ and $\plant$ are labeled so that the leading principal minor of $A_\obs$ of order $k$ is equal to $A_\plant$. Suppose there exists a matrix $M \in \mathbb{S}^n$ with the following properties:
\begin{enumerate}
\item $M_{i,j} = (A_\obs)_{i,j} ~ \mathrm{if}~ (A_\obs)_{i,j} = 1 ~\mathrm{or~if} ~ i= j$,
\item $M \in \mathrm{relint}\left(\mathcal{N}_{\sh([A_\plant - \gamma I_k]_{k \rightarrow n})}([A_\plant - \gamma I_k]_{k \rightarrow n})\right).$
\end{enumerate}
Then the Schur-Horn relaxation succeeds at identifying the planted subgraph $\plant$ inside the larger graph $\obs$, i.e., the unique optimal solution of the convex program (\ref{PrimalOptimization}) is $\hat{A}_{sh} = \begin{pmatrix} A_\plant - \gamma I_k & 0 \\ 0 & 0 \end{pmatrix}$.
\end{lemma}
\begin{proof}
From standard results in convex analysis  \cite{rockafellar2015convex}, we have that $\begin{pmatrix} A_\plant - \gamma I_k & 0 \\ 0 & 0 \end{pmatrix}$ is the unique optimal solution of (\ref{PrimalOptimization}) if $A_\obs$ can be decomposed as $A_\obs  \in K+ \mathrm{relint}\left(\mathcal{N}_{\sh([A_\plant - \gamma I_k]_{k \rightarrow n})}([A_\plant - \gamma I_k]_{k \rightarrow n})\right)$ for some matrix $K \in \Sym^n$ that satisfies:
\begin{equation*}
K_{i,j} = 0 ~ \mathrm{if~either}~ (A_\obs)_{i,j} = 1 ~\mathrm{or}~ i=j.
\end{equation*}
Letting $K = A_\obs-M$ we have the desired result.
\end{proof}

The assumption on the node labeling is made purely for the sake of notational convenience in our analysis (to avoid clutter in having to keep track of additional permutations), and our algorithmic methodology does not rely on such a labeling.  Based on this characterization of the optimality conditions, the success of the Schur-Horn relaxation relies on the existence of a suitable dual variable $M \in \mathbb{S}^n$ that satisfies two conditions.  The first of these conditions relates to the structure of the noise edges in $\obs$, while the second condition relates to the structure of the planted graph $\plant$ via the normal cone $\mathcal{N}_{\sh([A_\plant - \gamma I_k]_{k \rightarrow n})}([A_\plant - \gamma I_k]_{k \rightarrow n})$.  From the viewpoint of Lemma \ref{First Optimality Lemma}, favorable problem instances for the Schur-Horn relaxation are, informally speaking, those in which there are not too many noise edges in $\obs$ (implying a less restrictive first requirement on $M$) and in which the normal cone $\mathcal{N}_{\sh([A_\plant - \gamma I_k]_{k \rightarrow n})}([A_\plant - \gamma I_k]_{k \rightarrow n})$ is large (entailing a more flexible second condition for $M$).  The interplay between these two conditions forms the basis of our analysis and results presented in Section \ref{SectionRecovering}. In the remainder of the present section, we investigate spectral properties of planted graphs $\plant$ that result in a large normal cone $\mathcal{N}_{\sh([A_\plant - \gamma I_k]_{k \rightarrow n})}([A_\plant - \gamma I_k]_{k \rightarrow n})$.

The normal cones at the extreme points of the Schur-Horn orbitope are conveniently described based on the following notion (see Proposition \ref{Spect Comonotone Prop} in the sequel):
\begin{definition}
A pair of symmetric matrices $A,B \in \mathbb{S}^n$ is \emph{spectrally comonotone} if there exists an orthogonal matrix $U \in \R^{n \times n}$ such that $U' A U$ and $U' B U$ are both diagonal matrices with the diagonal entries sorted in nonincreasing order.
\end{definition}

The stipulation that two matrices be spectrally comonotone is a stronger condition than the requirement that the matrices be simultaneously diagonalizable, due to the additional restriction on the ordering of the diagonal entries upon conjugation by an orthogonal matrix.

\begin{example}
Consider the matrices $A = \begin{pmatrix} 3 & 0 & 0 \\ 0 & 1 & 0 \\ 0 & 0 & 1 \end{pmatrix}, B = \begin{pmatrix} 1 & 0 & 0 \\ 0 & 0.5 & 0.5 \\ 0 & 0.5 & 0.5 \end{pmatrix}, C = \begin{pmatrix} 1 & 0 & 0 \\ 0 & 1 & 1 \\ 0 & 1 & 1 \end{pmatrix}$.  The matrices $A$ and $B$ are spectrally comonotone, while $A$ and $C$ are only simultaneously diagonalizable and are not spectrally comonotone.
\end{example}

As Proposition \ref{First Optimality Lemma} states the optimality conditions of the Schur-Horn relaxation in terms of the \emph{relative interiors} of normal cones at extreme points of the Schur-Horn orbitope, we need the following ``strict'' analog of spectral comonotonicity:

\begin{definition}
A matrix $A\in \mathbb{S}^n$ is \emph{strictly spectrally comonotone} with a matrix $B\in\mathbb{S}^n$, if for every $P \in \mathbb{S}^n$ that is simultaneously diagonalizable with $B$, there exists $\epsilon>0$ such that $A + \epsilon \,P$ and $B$ are spectrally comonotone.
\end{definition}

Strict spectral comonotonicity is more restrictive than spectral comonotonicity.  Further, the definition of strict spectral comonotonocity is not a symmetric one, unlike that of spectral comonotonicity, i.e., even if $A \in \Sym^n$ is strictly spectrally comonotone with $B \in \Sym^n$, it may be that $B$ is not strictly spectrally comonotone with $A$.

\begin{example}
Consider the matrices $A = \begin{pmatrix} 3 & 0 & 0 \\ 0 & 2 & 0 \\ 0 & 0 & 1 \end{pmatrix}, B = \begin{pmatrix} 3 & 0 & 0 \\ 0 & 1 & 0 \\ 0 & 0 & 1 \end{pmatrix}$.  The matrix $A$ is strictly spectrally comonotone with the matrix $B$, but $B$ is not strictly spectrally comonotone with $A$.
\end{example}

The following result provides a characterization of normal cones at extreme points of the Schur-Horn orbitope in terms of spectrally comonotone matrices:

\begin{proposition} \label{Spect Comonotone Prop}
For any matrix $M \in \mathbb{S}^n$ and the associated Schur-Horn orbitope $\sh(M)$, the normal cone $\mathcal{N}_{\sh(M)}(W)$ and its relative interior at an extreme point $W$ of $\sh(M)$ are given by:
\begin{align*}
\mathcal{N}_{\sh(M)}(W) &= \{Q \in \mathbb{S}^n~|~ Q ~\mathrm{and}~ W~\mathrm{are~spectrally~comonotone} \}.\\
\mathrm{relint}\,\big(\mathcal{N}_{\sh(M)}(W)\big) &= \{Q \in \mathbb{S}^n~|~ Q ~\mathrm{is~strictly~spectrally~comonotone~with}~ W \}.
\end{align*}
\end{proposition}

\paragraph{Note} For any matrix $M \in \mathbb{S}^n$, the extreme points of $\sh(M)$ are the elements of the set $\{U M U' ~|~ U \in \mathcal{O}_n\}$, as each of the matrices $U M U'$ for $U \in \mathcal{O}_n$ has the same Frobenius norm.

\begin{proof}
Let $W = M$ without loss of generality.  We have that:
\begin{align*}
\mathcal{N}_{\sh(M)}(M) \,=&\, \{Y \in \Sym^n ~|~ \sup_{Z \in \sh(M)} ~ \tr(Y Z) \,\leq \, \tr(Y M)\} \\ \,=&\, \{Y \in \Sym^n ~|~ \sup_{Z = U M U' ~\mathrm{for}~ U \in \mathcal{O}_n} ~ \tr(Y Z) \,\leq \, \tr(Y M)\} \\ \,=&\, \{Y \in \Sym^n ~|~ \sup_{U \in \mathcal{O}_n} ~ \tr(U' Y U M) \,= \, \tr(Y M)\}.
\end{align*}
The last line follows from the inequality $\tr(Y M) \leq \sup_{U \in \mathcal{O}_n} ~ \tr(U' Y U M)$.  Considering the case of equality in the Von Neumann trace inequality \cite{neumann1937some}, we have that $\sup_{U \in \mathcal{O}_n} ~ \tr(U' Y U M) \,= \, \tr(Y M)$ if and only if $Y$ and $M$ are spectrally comonotone.  The claim about the relative interior of the normal cone follows immediately from the definition of strict spectral comonotonicity.
\end{proof}

If a matrix $M \in \mathbb{S}^n$ has few distinct eigenvalues, the normal cone at an extreme point $U M U'$ (for $U$ orthogonal) of $\sh(M)$ is larger as there are many more matrices that are spectrally comonotone with $U M  U'$.  Based on Proposition \ref{Spect Comonotone Prop}, this observation suggests that planted graphs $\plant$ with \emph{few distinct eigenvalues} have large normal cones $\mathcal{N}_{\sh([A_\plant - \gamma I_k]_{k \rightarrow n})}([A_\plant - \gamma I_k]_{k \rightarrow n})$, and such graphs are especially amenable to recovery in planted subgraph problems via the Schur-Horn relaxation. We make this insight more precise with our analysis in Section \ref{Subsection MainResult}.  Proposition \ref{Spect Comonotone Prop} also points to the utility of employing the parameter $\gamma$ in the Schur-Horn relaxation \eqref{PrimalOptimization}.  Specifically, multiplicities in the spectrum of the matrix $[A_\plant - \gamma I_k]_{k \rightarrow n} \in \mathbb{S}^n$ may be increased via suitable choices of $\gamma$, which in turn makes the normal cone $\mathcal{N}_{\sh([A_\plant - \gamma I_k]_{k \rightarrow n})}([A_\plant - \gamma I_k]_{k \rightarrow n})$ larger.  In particular, setting $\gamma$ equal to an eigenvalue of $A_\plant$ increases the multiplicity of zero as an eigenvalue of $[A_\plant - \gamma I_k]_{k \rightarrow n}$.  As detailed in Section \ref{SectionRecovering}, the success of the Schur-Horn relaxation relies on the existence of an eigenspace $\eig \subset \R^k$ of $A_\plant$ with small coherence parameter, and the appropriate choice of $\gamma$ is the eigenvalue $\lambda_\eig$ associated to $\eig$.  In our experiments in Section \ref{SectionNumerical}, we set $\gamma$ equal to the eigenvalue of $A_\plant$ with largest multiplicity, so that the multiplicity of zero as an eigenvalue of $[A_\plant - \gamma I_k]_{k \rightarrow n}$ is as large as possible.

To conclude, we record an observation on spectral comonotonocity that is useful in Section \ref{SectionRecovering}.  The claim is straightforward and therefore we omit the proof.
\begin{lemma} \label{Lemma strictly spectrally comonotone}
A pair of symmetric matrices $A,B \in \mathbb{S}^n$ is spectrally comonotone if and only if $A$ and $B$ are simultaneously diagonalizable and
\begin{align}
\lambda_{\mathrm{min}}(A|_{\eig_i}) \geq \lambda_{\mathrm{max}}(A|_{\eig_{i+1}}) \quad \forall i \in \{1,\,\dots,\,t-1\}, \label{spectineq}
\end{align}
where $\eig_i$ for $i \in \{1,\,\dots,\,t\} $ are eigenspaces of $B$ ordered such that the corresponding eigenvalues of $B$ are decreasing.  Further, $A$ is strictly spectrally comonotone with $B$ if and only if $A$ and $B$ are simultaneously diagonalizable and each of the inequalities \eqref{spectineq} holds strictly.
\end{lemma}
Note that if $A$ and $B$ simultaneously diagonalizable, then any eigenspace $\eig$ of $B$ is an invariant subspace of $A$.  As a result, the restriction of $A$ to the eigenspaces of $B$ in \eqref{spectineq} is consistent with the notation described in Section \ref{Outline}.

\subsection{Graphs with Few Eigenvalues} \label{SubsectionGraphsFewEvals}

Building on the preceding section, we give examples of families of graphs consisting of few distinct eigenvalues.  Such graphs have received much attention due to their connections to topics in combinatorics and design theory such as pseudorandomness \cite{krivelevich2006pseudo} and association schemes \cite{bannai1984algebraic,godsil1993algebraic}.

\begin{figure}[hbt]
\centering
   \subcaptionbox{\label{fig:Triangular8}}{\includegraphics[scale=0.25]{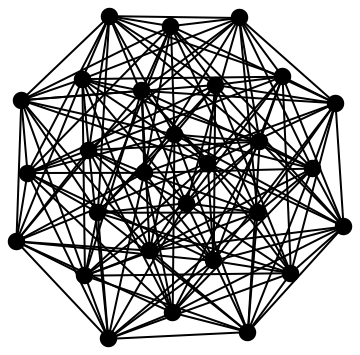}} \hspace{0.5in}%
   \subcaptionbox{\label{fig:Triangular9}}{\includegraphics[scale=0.25]{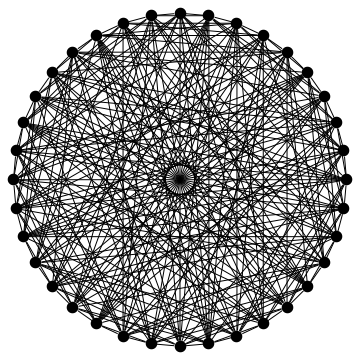}} \hspace{0.5in}%
   \subcaptionbox{\label{fig:PetersenGraph}}{\includegraphics[scale=0.25]{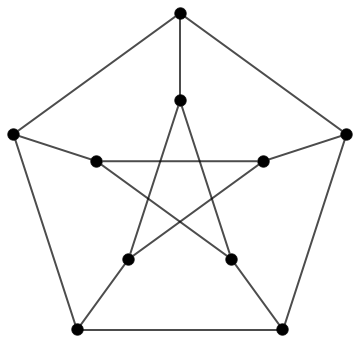}}
   \caption{From left to right: $8$-triangular graph, $9$-triangular graph, and Petersen graph.}\label{TriangularPetersenGraphs}
\end{figure}

\paragraph{Triangular graphs} The triangular graph $\mathcal{T}_m$ of order $m$ is the line graph of the complete graph on $m$ nodes.  The graph $\mathcal{T}_m$ has ${m \choose 2}$ nodes and it has the three distinct eigenvalues $2(m-2)$ (with multiplicity $1$), $m-4$ (with multiplicity $m-1$), and $-2$ (with multiplicity $\frac{m(m-3)}{2}$).  Figure \ref{TriangularPetersenGraphs} gives two examples.

\paragraph{Kneser graphs} A Kneser graph $\mathcal{K}(m,\ell)$ is a graph on ${m \choose \ell}$ nodes, each corresponding to an $\ell$-element subset of $m$ elements, and it consists of edges between those pairs of vertices for which the corresponding subsets are disjoint.  The graph $\mathcal{K}(m,1)$ is the complete graph on $m$ nodes and the graph $\mathcal{K}(5,2)$ is the Petersen graph (Figure \ref{TriangularPetersenGraphs}).  The Kneser graph $\mathcal{K}(m,\ell)$ has $\ell+1$ distinct eigenvalues in general.

\begin{figure}[hbt]
\centering
   \subcaptionbox{\label{fig:Paley5}}{\includegraphics[scale=0.25]{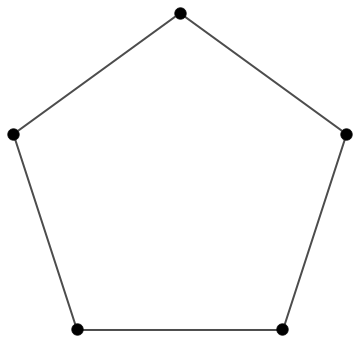}} \hspace{.5in}%
   \subcaptionbox{\label{fig:Paley13}}{\includegraphics[scale=0.25]{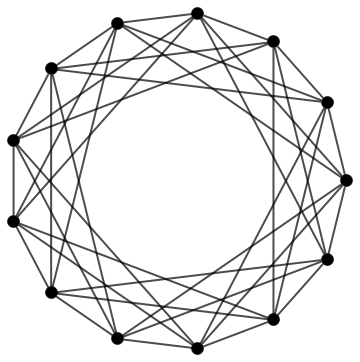}}\hspace{.5in}%
   \subcaptionbox{\label{fig:Paley17}}{\includegraphics[scale=0.25]{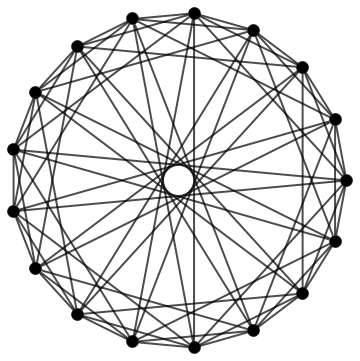}}
   \caption{From left to right: $5$-Paley graph, $13$-Paley graph, and $17$-Paley graph.}\label{PaleyGraph}
\end{figure}

\paragraph{Paley graphs} Let $q$ be a prime power such that $q \equiv 1 (\mathrm{mod}~ 4)$.  The Paley graph on $q$ nodes is an undirected graph formed by connecting pairs of nodes $i,j \subset \{0,\dots,q-1\}$ if the difference $i-j$ is a square in the finite field $\mathrm{GF}(q)$.  Note that $i-j$ is a square if and only if $j-i$ is a square as $-1$ is a square in $\mathrm{GF}(q)$.  Paley graphs have eigenvalues $\frac{1}{2}(q-1)$ (with multiplicity $1$), $\frac{1}{2}(-1+\sqrt{q})$ (with multiplicity $\frac{1}{2}(q-1)$), and $\frac{1}{2}(-1-\sqrt{q})$ (with multiplicity $\frac{1}{2}(q-1)$).  Paley graphs are also examples of pseudorandom graphs as they exhibit properties similar to random graphs (in the limit of large $q$) \cite{krivelevich2006pseudo}.  Figure \ref{PaleyGraph} shows the three smallest Paley graphs.

\begin{figure}[hbt]
\centering
   \subcaptionbox{\label{fig:GQ22}}{\includegraphics[scale=0.25]{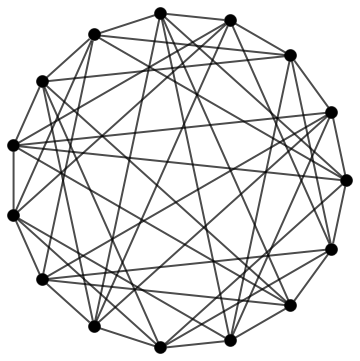}} \hspace{0.5in}%
   \subcaptionbox{\label{fig:GQ24}}{\includegraphics[scale=0.25]{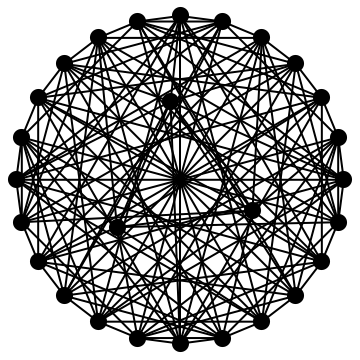}}
\caption{Generalized quadrangle-$(2,2)$ graph (left) and generalized quadrangle-$(2,4)$ graph (right).}\label{GeneralizedQuadrangleGraph}
\end{figure}

\paragraph{Strongly regular graphs} These are regular graphs with the property that every pair of adjacent vertices has the same number $d_a$ of common neighbors and every pair of non-adjacent vertices has the same number $d_{na}$ of common neighbors, for some integers $d_a, d_{na}$ \cite{bose1963strongly}.  Strongly regular graphs that are connected have three distinct eigenvalues; conversely, connected and regular graphs with three distinct eigenvalues are necessarily strongly regular. The triangular graphs, Kneser graphs with parameter $\ell=2$ and the Paley graphs mentioned above are examples of strongly regular graphs. The Clebsch graph shown in Figure \ref{ClebschGraph} in the introduction is also a strongly regular graph with degree $5$ and eigenvalues $5$ (with multiplicity $1$), $-3$ (with multiplicity $5$), and $1$ (with multiplicity $10$).  The generalized quadrangle graphs shown in Figure \ref{GeneralizedQuadrangleGraph} are additional examples of strongly regular graphs. Strongly regular graphs form a significant topic in graph theory due to their many regularity properties \cite{brouwer1984strongly,cameron1978strongly,seidel1969strongly}.

\begin{figure}[hbt]
\centering
   \subcaptionbox{\label{fig:Hamming33}}{\includegraphics[scale=0.25]{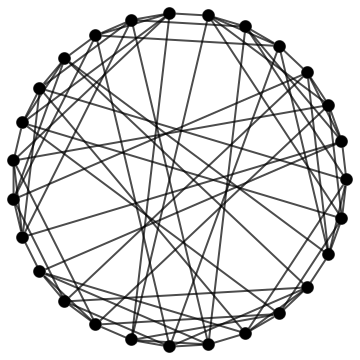}} \hspace{0.5in}%
   \subcaptionbox{\label{fig:Hypercube6}}{\includegraphics[scale=0.25]{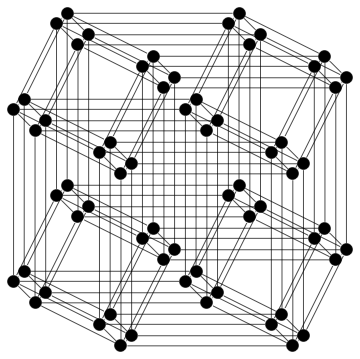}}
\caption{Hamming-$(3,3)$ graph (left) and $6$-hypercube graph (right).}\label{HammingScheme}
\end{figure}

\paragraph{Other examples} Unlike regular graphs with three distinct eigenvalues, graphs with four (or more) eigenvalues do not appear to have a simple combinatorial characterization \cite{van1996graphs}.  Nonetheless, there are many constructions of such graphs in the literature \cite{haemers1996spreads,van1995regular, van1996graphs}, most notably those derived from distance-regular graphs \cite{brouwer2012distance} and from association schemes.  Graphs from association schemes of class-$d$ have at most $d+1$ eigenvalues, and therefore several examples of graphs with four eigenvalues come from $3$-class association schemes \cite{chang1994imprimitive,mathon19753}. The two graphs shown in Figure \ref{HammingScheme} are  obtained from the Hamming scheme.

\section{Recovering Subgraphs Planted in Erd\H{o}s-R\'{e}nyi Random Graphs} \label{SectionRecovering}

In this section we discuss our theoretical results on the performance of the Schur-Horn relaxation in recovering subgraphs planted inside Erd\H{o}s-R\'{e}nyi random graphs.  Formally, suppose without loss of generality as in the previous section that the nodes of $\obs$ and of $\plant$ are labeled so that the leading principal minor of $A_\obs$ of order $k$ is equal to $A_\plant$.  The Erd\H{o}s-R\'{e}nyi model for the planted subgraph problem specifies a distribution on the edges in the remainder of the graph $\obs$ via a probability parameter $p \in [0,1]$; for each $i, j \in \{1,\dots,n\}$ with $i < j$ and $k < j$, the graph $\obs$ contains an edge between nodes $i$ and $j$ with probability $p$ (independent of the other edges):
\begin{equation*}
(A_\obs)_{i,j} = (A_\obs)_{j,i} = \begin{cases} 1, ~ \mathrm{with~probability}~ p, \\ 0, ~ \mathrm{with~probability}~ 1-p. \end{cases}
\end{equation*}
We begin with a sufficient condition for the optimality condition described in Lemma \ref{First Optimality Lemma}, which suggests a natural approach for constructing suitable dual variables for certifying optimality.  These sufficient conditions point to the importance of the existence of an eigenspace of $A_\plant$ with certain properties to the success of the Schur-Horn relaxation; these properties are discussed in Section \ref{Subsection CentralNotions}. In Section \ref{Subsection MainResult} we state and prove the main theorem (Theorem \ref{Main Theorem}) of this paper, with Section \ref{Subsection Simplifications} giving specializations of this result (e.g., to the planted clique problem).

\subsection{A Simpler Sufficient Condition for Optimality}
The following proposition provides a simpler set of conditions than those in Lemma \ref{First Optimality Lemma} on dual variables that certify the success of the Schur-Horn relaxation.  This result continues to be deterministic in nature, and the probabilistic aspects of our analysis -- due to the Erd\H{o}s-R\'{e}nyi model -- appear in the sequel.

\begin{proposition} \label{Optimality Conditions}
Consider a planted subgraph problem instance in which the nodes of $\obs$ and $\plant$ are labeled so that the leading principal minor of $A_\obs$ of order $k$ is equal to $A_\plant$. Suppose there exists an eigenspace $\eig\subset \R^k$ of $A_\plant$ with eigenvalue $\lambda_\eig$, and suppose there exists a matrix $M = \begin{pmatrix} M_{11} & M_{12} \\  M_{12}' & M_{22} \end{pmatrix} \in \mathbb{S}^n$ with submatrices $M_{11} \in \mathbb{S}^k, M_{12} \in \R^{k \times (n-k)}, M_{22} \in \mathbb{S}^{n-k}$ such that the the following conditions are satisfied:
\begin{enumerate}
\item[$(i)$] $M_{i,j} = (\Ag)_{i,j}$, if $(\Ag)_{i,j}= 1$ or if $i=j$,\label{Optimality Condition E constraint}

\item[$(ii)$] The submatrix $M_{11} \in \mathbb{S}^k$ is strictly spectrally comonotone with $A_\plant$,	\label{Optimality Condition M11Comonotone}

\item[$(iii)$] $\lambda_{\mathrm{max}}(M_{11}|_{\eig}) \geq \lambda_\eig$ and $\lambda_{\mathrm{min}}(M_{11}|_{\eig}) \leq \lambda_\eig$, \label{Optimality Condition M11 restricted}
\item[$(iv)$] Each column of the submatrix $M_{12} \in \R^{k \times (n-k)}$ lies in the subspace $\eig$,\label{Optimality Condition M12Columns}
\item[$(v)$] $\egap(M_{11},\eig)>\|M_{12}\|_2 + \|M_{22}\|_2 + |\lambda_\eig|$. \label{Optimality Condition Eigengap}
\end{enumerate}
Then the Schur-Horn relaxation (\ref{PrimalOptimization}) with parameter $\gamma = \lambda_\eig$ succeeds at identifying the planted subgraph $\plant$ inside the larger graph $\obs$.
\end{proposition}

\begin{proof}
We establish this result by showing that the given matrix $M \in \Sym^n$ satisfies the requirements of Lemma $\ref{First Optimality Lemma}$.  The first condition of Lemma \ref{First Optimality Lemma} is identical to that of this proposition, and therefore it is satisfied.  We prove next that the remaining conditions of this proposition ensure that the second requirement of Lemma \ref{First Optimality Lemma} is also satisfied, i.e., $M \in \mathrm{relint}\left(\mathcal{N}_{\sh([A_\plant - \lambda_\eig I_k]_{k \rightarrow n})}([A_\plant - \lambda_\eig I_k]_{k \rightarrow n})\right)$.  Based on Proposition \ref{Spect Comonotone Prop}, this entails showing that $M$ is strictly spectrally comonotone with $[A_\plant - \lambda_\eig I_k]_{k \rightarrow n}$.  Our strategy is to employ Lemma \ref{Lemma strictly spectrally comonotone}.

Let $\eig_i \subset \R^k, ~ i=1,\dots,t$ be the eigenspaces of $A_\plant$ ordered such that the corresponding eigenvalues $\lambda_{\eig_i}$ are strictly decreasing, and suppose $\eig_j = \eig, ~ \lambda_{\eig_j} = \lambda_\eig$ for some $j \in \{1,\dots,t\}$.  As $0$ is an eigenvalue of $A_\plant - \lambda_\eig I_k$, one can check that the eigenspaces of $[A_\plant - \lambda_\eig I_k]_{k \rightarrow n}$ are $\tilde{\eig}_i = \eig_i \times \{0\} \subset \R^k \times \R^{n-k}, ~ i=1,\dots,t, ~ i \neq j$ (with corresponding eigenvalues $\lambda_{\eig_i} - \lambda_\eig$) and $\tilde{\eig}_j = \eig \times \R^{n-k} \subset \R^k \times \R^{n-k}$ (with eigenvalue $0$).  We now need to show that $M$ and $[A_\plant - \lambda_\eig I_k]_{k \rightarrow n}$ are simultaneously diagonalizable, and that $\lambda_{\mathrm{min}}(M|_{\tilde{\eig}_i}) > \lambda_{\mathrm{max}}(M|_{\tilde{\eig}_{i+1}})$ for $i \in \{1,\,\dots,\,t-1\}$.

First, as $\eig$ is an eigenspace of $A_\plant - \lambda_\eig I_k$ with eigenvalue $0$ and as every column of $M_{12}$ belongs to $\eig$, one can check that $(A_\plant - \lambda_\eig I_k) \cdot M_{12} = 0 \in \R^{k \times (n-k)}$.  Further, from Lemma \ref{Lemma strictly spectrally comonotone} we note that $M_{11}$ and $A_\plant - \lambda_\eig I_k$ are simultaneously diagonalizable because $M_{11}$ is strictly spectrally comonotone with $A_\plant$ (and hence with  $A_\plant - \lambda_\eig I_k$).  From these two observations one can check that $M$ and $[A_\plant - \lambda_\eig I_k]_{k \rightarrow n}$ commute with each other, and therefore are simultaneously diagonalizable.

As $M$ and $[A_\plant - \lambda_\eig I_k]_{k \rightarrow n}$ are simultaneously diagonalizable, we have that the eigenspaces $\tilde{\eig}_i, ~ i=1,\dots,t$ of $[A_\plant - \lambda_\eig I_k]_{k \rightarrow n}$ are invariant subspaces of $M$.  Similarly, as $M_{11}$ is strictly spectrally comonotone with $A_\plant$, the eigenspaces $\eig_i$ are invariant subspaces of $M_{11}$.  Based on the structure of these eigenspaces as described above, one can check that the eigenvalues of $M|_{\tilde{\eig}_i}$ are equal to those of $M_{11}|_{\eig_i}$ for each $i = 1,\dots,t, ~ i \neq j$.  Hence, $\lambda_{\min}(M|_{\tilde{\eig}_i}) > \lambda_{\max}(M|_{\tilde{\eig}_{i+1}})$ for $i > j$ and for $i < j-1$.

All that remains to be verified is that $\lambda_{\min}(M|_{\tilde{\eig}_j}) > \lambda_{\max}(M|_{\tilde{\eig}_{j+1}})$ and that $\lambda_{\min}(M|_{\tilde{\eig}_{j-1}}) > \lambda_{\max}(M|_{\tilde{\eig}_{j}})$. As each column of $M_{12}$ belongs to $\eig$ and as $\tilde{\eig}_j = \eig \times \R^{n-k} \subset \R^k \times \R^{n-k}$, we have for $x \in \eig, y \in \R^{n-k}$ that:
\begin{equation}
M|_{\tilde{\eig}_j} \begin{pmatrix}x \\ y \end{pmatrix} = \left[\begin{pmatrix} M_{11}|_{\eig} & 0 \\ 0  & 0 \end{pmatrix} + \begin{pmatrix} 0 & M_{12} \\ M_{12}'  & M_{22} \end{pmatrix} \right] \begin{pmatrix}x \\ y \end{pmatrix} = \begin{pmatrix} M_{11}|_{\eig} x + M_{12} y \\ M_{12}' x + M_{22} y \end{pmatrix} \in \tilde{\eig}_j. \label{structurewithineig}
\end{equation}
Consequently, recalling that $\eig_j = \eig$ we have:
\begin{equation*}
\begin{aligned}
\lambda_{\max}(M|_{\tilde{\eig}_j}) &\leq \max\{\lambda_{\max}(M_{11}|_{\eig}),0\} + \|M_{12}\|_2 + \|M_{22}\|_2 \\ &< \max\{\lambda_{\max}(M_{11}|_{\eig}),0\} - |\lambda_\eig| + \egap(M_{11},\eig) \\ &\leq \max\{\lambda_{\max}(M_{11}|_{\eig}),0\} - |\lambda_\eig| + \lambda_{\min}(M_{11}|_{\eig_{j-1}}) - \lambda_{\max}(M_{11}|_{\eig}) \\ &= \max\{0,-\lambda_{\max}(M_{11}|_{\eig})\} - |\lambda_\eig| + \lambda_{\min}(M_{11}|_{\eig_{j-1}}) \\ &\leq \max\{0,-\lambda_{\eig}\} - |\lambda_\eig| + \lambda_{\min}(M_{11}|_{\eig_{j-1}}) \\ &\leq \lambda_{\min}(M_{11}|_{\eig_{j-1}}) \\ &= \lambda_{\min}(M|_{\tilde{\eig}_{j-1}}).
\end{aligned}
\end{equation*}
The first inequality follows from \eqref{structurewithineig}, the second inequality from condition $(v)$, the third inequality from the definition of $\egap$ (see Section \ref{Outline}) as $\eig_j = \eig$, the fourth inequality from condition $(iii)$, and the second equality from the fact that the eigenvalues of $M|_{\tilde{\eig}_i}$ are equal to those of $M_{11}|_{\eig_i}$ for each $i = 1,\dots,t, ~ i \neq j$.  Similarly, one can check that $\lambda_{\min}(M|_{\tilde{\eig}_j}) > \lambda_{\max}(M|_{\tilde{\eig}_{j+1}})$.  This concludes the proof.
\end{proof}

This result provides a concrete approach for constructing dual variables to certify the optimality of the Schur-Horn relaxation (\ref{PrimalOptimization}) at the desired solution.  In the remainder of this section, we give conditions on the eigenstructure of the planted graph $\plant$, the probability $p$ of the Erd\H{o}s-R\'{e}nyi model, and the size $n$ of the larger graph $\obs$ under which the Schur-Horn relaxation (\ref{PrimalOptimization}) succeeds with high probability.

\subsection{Invariants of Graph Eigenspaces} \label{Subsection CentralNotions}

In this section, we investigate properties of eigenspaces of graphs which ensure that the conditions of Proposition \ref{Optimality Conditions} can be satisfied.  For notational clarity in the discussion in this section, we let $\Omega_j \subset \{1,\dots,k\}$ for $j=1,\dots,n-k$ denote the locations of the entries equal to one in the submatrix $(A_\obs)_{i,j+k}, ~ i=1,\dots,k; j=1,\dots,n-k$, i.e., $(A_\obs)_{i,j+k} = 1 \Leftrightarrow i \in \Omega_j$.

A requirement of Proposition \ref{Optimality Conditions} is the existence of a suitable eigenspace $\eig \subset \R^k$ of $A_\plant$ such that one can obtain a matrix $M_{12} \in \R^{k \times (n-k)}$ (a submatrix of a larger dual certificate) that satisfies three conditions:  $(i)$ Every column of $M_{12}$ lies in $\eig$, $(ii)$ For each $i=1,\dots,k$ and $j=1,\dots,n-k$ we have that $(M_{12})_{i,j} = 1$ if $(A_\obs)_{i,j+k} = 1$, and $(iii)$ The operator norm $\|M_{12}\|_2$ is as small as possible.

We begin by analyzing the first two conditions and the restrictions they impose on $\eig$.  Consider the $j$'th column of $M_{12}$ for a fixed $j \in \{1,\dots,n-k\}$ as an illustration.  Then conditions $(i)$ and $(ii)$ are simultaneously satisfied if the coordinate subspace of vectors in $\R^k$ with support on the indices in $\Omega_j$ has a \emph{transverse intersection} with $\eig^\perp$.  More generally, a natural sufficient condition for the first two requirements on $M_{12}$ to be satisfied (for every column) is for $\eig^\perp$ to have a transverse intersection with the coordinate subspaces specified by each of the subsets $\Omega_j$ for $j=1,\dots,k$.  This observation leads to the following invariant that characterizes the transversality of a subspace with all coordinate subspaces of a certain dimension:

\begin{definition} \cite{kruskal1977three}
The \emph{Kruskal rank} of a subspace $\mathcal{S} \subseteq \R^k$, denoted $\krank{\mathcal{S}}$, is the largest $m \in \mathbb{Z}$ such that for any $\Omega \subseteq \{1,\dots,k\}$ with $|\Omega| = m$ we have:
\begin{equation*}
\mathcal{S}^\perp \cap \{v \in \R^k ~|~ v_i = 0 ~\mathrm{if}~ i \notin \Omega\} = \{0\}.
\end{equation*}
\end{definition}
In other words, the Kruskal rank of a subspace $\mathcal{S} \subset \R^k$ is one less than the size of the support of the sparsest nonzero vector in $\R^k$ that is orthogonal to $\mathcal{S}$.  The Kruskal rank of a matrix -- the largest $m$ such that all subsets of $m$ columns of the matrix are linearly independent -- was first introduced in \cite{kruskal1977three} in the context of tensor decompositions.  This version in terms of matrices is equivalent to our definition in terms of subspaces.  One can check that all principal minors of $\proj_{\mathcal{S}}$ of size upto $\krank{\mathcal{S}}$ are non-singular.

Recall that the entries $(A_\obs)_{i,j+k}$ for $i=1,\dots,k$ and $k=1,\dots,n-k$ correspond to edges (or lack thereof) between nodes in $\obs$ outside the induced subgraph corresponding to $\plant$ and those of $\plant$.  Therefore, if we employ the Schur-Horn relaxation with parameter $\gamma = \lambda_\eig$ (the eigenvalue associated to $\eig$), then the Kruskal rank of $\eig$ provides a bound on the number of noise edges that can be tolerated between these two sets of nodes.  As such $\krank{\eig}$ plays a central role in our main result (see Theorem \ref{Main Theorem}) in providing an upper bound on the probability of a noise edge in $\obs$ under the Erd\H{o}s-R\'{e}nyi model.

Returning to the three conditions on $M_{12}$ stated at the beginning of this section, if an eigenspace $\eig$ of $A_\plant$ has large Kruskal rank and if the size of each $\Omega_j, ~ j=1,\dots,n-k$ is smaller than $\krank{\eig}$, then there is an affine space (of dimension potentially larger than zero) of matrices in $\R^{k \times (n-k)}$ that satisfy the first two requirements on $M_{12}$.  The third condition on $M_{12}$ requires that we find the element of this affine space with the smallest spectral norm:
\begin{equation*}
\begin{aligned}
\hat{M}_{12}^{\mathrm{spectral}} = \argmin_{X \in \R^{k \times (n-k)}} ~ & \|X\|_2 \\\mathrm{s.t.} ~ & X_{i,j} = 1 ~\mathrm{if}~ i \in \Omega_j ~\mathrm{for}~ j=1,\dots,n-k \\ & \mathrm{col}(X) \subseteq \eig.
\end{aligned}
\end{equation*}
As long as $|\Omega_j| \leq \krank{\eig}$ for each $j=1,\dots,n-k$, this problem is feasible.  However, analytically characterizing the optimal value and solution of this problem is challenging, especially in the context of problem instances that arise from the Erd\H{o}s-R\'{e}nyi model, as the subsets $\Omega_j, ~ j=1,\dots,n-k,$ are random.  As a result, a common approach is to replace the objective in the above problem with the Frobenius norm:
\begin{equation} \label{dual12}
\begin{aligned}
\hat{M}_{12}^{\mathrm{frobenius}} = \argmin_{X \in \R^{k \times (n-k)}} ~ & \|X\|_F \\\mathrm{s.t.} ~ & X_{i,j} = 1 ~\mathrm{if}~ i \in \Omega_j ~\mathrm{for}~ j=1,\dots,n-k \\ & \mathrm{col}(X) \subseteq \eig.
\end{aligned}
\end{equation}
One of the virtues of this latter formulation in comparison to the earlier one is that the spectral norm of the optimal solution $\|\hat{M}_{12}^{\mathrm{frobenius}}\|_2$ is more tractable to bound, primarily since the optimization problem \eqref{dual12} decomposes into $n-k$ separable problems, one for each column of the decision variable $X$.  In particular, for any subspace $\mathcal{S} \subseteq \R^k$ and any $\Omega \subset \{1,\dots,k\}$ with $|\Omega| \leq \krank{\mathcal{S}}$, consider the following minimum Euclidean-norm completion:
\begin{equation} \label{completion}
\begin{aligned}
q_\Omega(\mathcal{S}) &\triangleq \argmin_{q \in \R^k} ~ \|q\| ~~~ \mathrm{s.t.} ~ q \in \mathcal{S} ~\mathrm{and}~ q_i = 1 ~\mathrm{for}~ i \in \Omega \\ &= \proj_{\mathcal{S}} {I_{\Omega}}' (({\proj_{\mathcal{S}}})_{{\Omega},{\Omega}})^{-1}{1_{|{\Omega}|}}.
\end{aligned}
\end{equation}
With this notation, the $j$'th column of $\hat{M}_{12}^{\mathrm{frobenius}}$ is given by $q_{\Omega_j}(\eig)$.  Further, under the Erd\H{o}s-R\'{e}nyi model, the entries $(A_\obs)_{i,j+k}, ~ i=1,\dots,k;j=1,\dots,n-k$ are independent and identically distributed Bernoulli random variables.  In such a family of problem instances, the columns of $\hat{M}_{12}^{\mathrm{frobenius}}$, i.e., $q_{\Omega_j}(\eig) \in \R^k, ~j=1,\dots,k$, are independently and identically distributed random vectors.  These observations in conjunction with the following tail bound on the spectral norm of a random matrix suggest a natural invariant of $\eig$ that leads to bounds on $\|\hat{M}_{12}^{\mathrm{frobenius}}\|_2$:

\begin{lemma} \cite{vershynin2012introduction} \label{Vershynin Theorem 5.44}
Let ${A}$ be a $d \times N$ matrix ($d<N$) with columns $A_i$ and let ${\Sigma} = \mathbb{E} [{A_i} {A_i}^T]$ denote the correlation matrix of the $A_i$'s.  Further, suppose there exists $m \in \R$ such that $\|A_i\| \leq \sqrt{m}$ almost surely for all $i$.  Then $\forall  x \geq (N \|\Sigma\|_2)^{1/2}$ we have that
\begin{equation}
		\mathbb{P}(\|A\|_2 \geq x) \leq 2d \exp{\Big(-\frac{3 ({x^2} - N \|\Sigma\|_2)^2}{4m({x^2} + 2 N \|\Sigma\|_2)}}\Big).
	\end{equation}	
\end{lemma}
\begin{proof}
The proof follows that of Theorem 5.41 in \cite{vershynin2012introduction} with minor modifications.  We apply the non-commutative Bernstein Inequality to $\frac{1}{N} x^2 - \|\Sigma\|_2$ rather than to $\max(\delta,\delta^2)$ on p.27 of \cite{vershynin2012introduction}, and we don't make the isotropy assumption.
\end{proof}

To apply Lemma \ref{Vershynin Theorem 5.44} to obtain a bound on $\|\hat{M}_{12}^{\mathrm{frobenius}}\|_2$, we describe next the second key invariant of $\eig$, which is essentially the correlation matrix in Lemma \ref{Vershynin Theorem 5.44}.

\begin{definition}
Let $\mathcal{S} \subseteq \R^k$ be a subspace.  Then the \emph{combinatorial width} of $\mathcal{S}$ for each $\ell = 1,\dots,\krank{\mathcal{S}}$ and $p \in [0,1)$ is defined as:
\begin{equation*}
\omega(\mathcal{S},\ell,p) \triangleq \Big\|\mathbb{E} [{q_\Omega(\mathcal{S})} \,{q_{\Omega}(\mathcal{S})}'\,\big|\, |\Omega|\leq \ell] \Big\|_2,
\end{equation*}
with the expectation taken over $\Omega$, where each element of $\{1,\dots,k\}$ is contained in $\Omega$ independently with probability $p$.
\end{definition}

The conditioning in the definition ensures that $q_\Omega(\mathcal{S})$ is well-defined as $|\Omega| \leq \krank{\mathcal{S}}$.  We utilize this terminology as a parallel to analogous notions such as `mean width' that are prominent in the convex geometry literature.  The explicit appearance of $\ell$ in this definition allows for a more fine-grained analysis in our main result Theorem \ref{Main Theorem}; see Section \ref{Subsection MainResult}.  Based on the following result, the Kruskal rank and the combinatorial width play a central role in Theorem \ref{Main Theorem} as the success of the Schur-Horn relaxation \eqref{PrimalOptimization} relies on the existence of an eigenspace $\eig$ of $A_\plant$ that has large Kruskal rank and small combinatorial width.

\begin{proposition} \label{M12 Proposition}
Consider a planted subgraph problem instance in which the nodes of $\obs$ and $\plant$ are labeled so that the leading principal minor of $A_\obs$ of order $k$ is equal to $A_\plant$, and the remaining edges in $\obs$ are drawn according to the Erd\H{o}s-R\'{e}nyi model with probability $p \in [0,\frac{\krank{\eig}}{k})$.  Fix any $\ell \in \mathbb{Z}$ satisfying $kp < \ell  \leq  \krank{\eig} $, and denote $\zeta := \min\limits_{\substack{\Omega \subset \{1,\dots,k\} \\ |\Omega| \leq \ell}} \lambda_{\mathrm{min}} \big((\proj_{\eig})_{\Omega,\Omega}\big)$. For any $\delta \geq  \sqrt{(n-k)\,\omega(\eig,\ell,p)}$, there exists a matrix $M_{12} \in \R^{k \times (n-k)}$ satisfying the following properties:
\begin{enumerate}
\item Each column of $M_{12}$ lies in $\eig$,\label{lemmaFirstCondition}
\item $(M_{12})_{i,j} \,=\,(A_\obs)_{i,\, j+k}$   if  $(A_\obs)_{i,\, j+k} = 1$, \label{lemmaSecondCondition}
\item $\|M_{12}\|_2 < \delta$,
\end{enumerate}
with probability at least $\Big(1-2k \exp{\big(-  \frac{3 \zeta (\delta^2 -(n-k)\, \omega(\eig,\ell,p))^2}  {4\, \ell(\delta^2 + 2(n-k) \omega(\eig,\ell,p))}\big) }\Big)\big(1- \exp{(-\frac{(\ell-kp)^2}{\ell + kp})}\big)^{n-k}$.
\end{proposition}
\begin{proof}
We bound the probability that $\hat{M}_{12}^{\mathrm{frobenius}}$ obtained as the optimal solution of \eqref{dual12} satisfies the requirements of this proposition.

We begin by bounding the cardinality of each $\Omega_j$ for $j=1,\dots,n-k$.  Under the Erd\H{o}s-R\'{e}nyi model, each $|\Omega_j|$ follows a binomial distribution.  Consequently, using the Chernoff bound we have for each $j=1,\dots,n-k$ that:
\begin{align*}
\mathbb{P}(|\Omega_j| \geq \ell+1) \leq \mathbb{P}(|\Omega_j| \geq \ell) = \mathbb{P}\Bigg(|\Omega_j| \geq \Big(1+\frac{\ell-kp}{kp}\Big)kp\Bigg) \leq \exp \Bigg(-\frac{(\ell-kp)^2}{\ell + kp}\Bigg).
\end{align*}
The first inequality is not essential and it is simply used to avoid notational clutter.  Based on the independence of the $\Omega_j$'s,
\begin{equation} \label{step1ofM12}
\mathbb{P}(|\Omega_j| \leq \ell, ~ j=1,\dots,n-k) \geq \Bigg(1- \exp\Big(-\frac{(\ell-kp)^2}{\ell + kp}\Big)\Bigg)^{n-k}.
\end{equation}
This inequality provides a bound on the probability that the optimization problem \eqref{dual12} is feasible.

In our next step we bound $\|\hat{M}_{12}^{\mathrm{frobenius}}\|_2$ via Lemma \ref{Vershynin Theorem 5.44}.  As $\ell \leq \krank{\eig}$ one can check that $\zeta > 0$.  Further, from \eqref{completion} we have that $\|q_{\Omega_j}\|^2 \leq \frac{|\Omega_j|}{\zeta}$.  Thus, by applying Lemma \ref{Vershynin Theorem 5.44}, we deduce that
\begin{align} \label{step2ofM12}
\mathbb{P}(\|M_{12}\|_2 < \delta ~\big|~ |\Omega_j| \leq \ell ~\forall j) \geq 1 - 2k \exp{\Bigg(-\frac{3 \zeta (\delta^2 -(n-k) \omega(\eig,\ell,p))^2}  {4\, \ell(\delta^2 + 2 (n-k) \omega(\eig,\ell,p))}\Bigg) }.
\end{align}
The final result follows by combining \eqref{step1ofM12} and \eqref{step2ofM12}.
\end{proof}


\subsection{Properties of Kruskal Rank and Combinatorial Width} \label{properties}

Beyond the utility of the Kruskal rank and combinatorial width in characterizing the performance of the Schur-Horn relaxation, these graph parameters are also of intrinsic interest and we discuss next their relationship to structural properties of $\plant$.

\subsubsection{Invariance under Complements for Regular Graphs} \label{subsubsec:regular}
Both the Kruskal rank and the combinatorial width are preserved under graph complements for regular graphs.  Suppose $\plant$ is a connected regular graph on $k$ vertices, and let $A_\plant \in \Sym^k$ be an adjacency matrix representing $\plant$ for some labeling of the nodes.  Then the eigenspaces of $A_\plant$ are the same as those of the adjacency matrix $A_{\plant^c}$ of the complement $\plant^c$ based on the following relation:
\begin{equation} \label{regularcomplement}
A_{\plant^c} = 1_k 1_k' - I_k - A_\plant.
\end{equation}
As $\plant$ is connected and regular, the vector $1_k$ is an eigenvector of $A_\plant$. Thus, the Kruskal ranks and the combinatorial widths associated to the eigenspaces of $A_\plant$ are the same as those associated to the eigenspaces of $A_{\plant^c}$.

\subsubsection{Combinatorial Width for Symmetric Graphs}
For graphs $\plant$ that are symmetric -- vertex- end edge-transitive -- and also have symmetric complements $\plant^c$, the combinatorial width of any eigenspace $\eig$ of $A_\plant$ can be characterized in terms of the minimum singular values of minors of $\proj_\eig$.  In particular, we establish our result by demonstrating that the correlation matrix $\mathbb{E}[{q_\Omega(\mathcal{S})} \,{q_{\Omega}(\mathcal{S})}'\,\big|\, |\Omega|\leq \ell]$ in the definition of the combinatorial width has the property that all its nonzero eigenvalues are equal to each other, which leads to bounds on the combinatorial width via bounds on the trace of the correlation matrix.

\begin{proposition} \label{Prop CombWidth Upperbound}
Let $A_\plant \in \Sym^k$ be an adjacency matrix of a (connected) symmetric graph $\plant$ with a symmetric complement $\plant^c$, and let $\eig \subset \R^k$ be an eigenspace of $A_\plant$.  Fix any $\ell \in \mathbb{Z}$ and $p \in [0,1)$ such that $kp \leq \ell \leq \krank{\eig}$, and let $\zeta := \min\limits_{\substack{\Omega \subset \{1,\dots,k\}  \\ |\Omega| \leq \ell}} \lambda_{\mathrm{min}} \big((\proj_{\eig})_{\Omega,\,\Omega}\big)$. Then,
\begin{align*}
\omega(\eig,\ell,p) \,\leq \,\frac{2 k p}{\zeta \dim(\eig)}.
\end{align*}
\end{proposition}

\begin{proof}
Denote the correlation matrix in the definition of the combinatorial width as follows:
\begin{align} \label{SigmaFirstEq}
{\Sigma} = \mathbb{E} \,[{q_\Omega(\eig)}{q_\Omega(\eig)}' \,\big|\, |\Omega| \leq \ell ] &= \sum\limits_{i=0}^{\ell} c_{p,\ell}\,p^i(1-p)^{k-i} \sum\limits_{|\Omega| = i} {q_\Omega(\eig)}{q_\Omega(\eig)}',
\end{align}
where the term $c_{p,\ell} = \big(\sum\limits_{i=0}^{\ell} {{k}\choose{i}}p^i(1-p)^{k-i}\big)^{-1}$ is the normalization constant.

The main element of the proof is to show that the rank of $\Sigma \in \Sym^k$ is equal to $\dim(\eig)$ (it is easily seen that $\mathrm{col}(\Sigma) \subseteq \eig$) and that all the nonzero eigenvalues of $\Sigma$ are equal to each other.  After this step is completed, one can bound the combinatorial width using the following relation:
\begin{equation}
\omega(\eig,\ell,p) = \frac{\tr(\Sigma)}{\dim(\eig)}. \label{widthtracerelation}
\end{equation}
In particular, we have $q_\Omega(\eig) = \proj_\eig {I_\Omega}' (({\proj_\eig})_{\Omega,\Omega})^{-1}{1_{|\Omega|}}$ with $|\Omega| \leq \krank{\eig}$.  One can check that $\|q_\Omega(\eig)\|^2 \leq \frac{|\Omega|}{\zeta}$, and then obtain that:
\begin{equation} \label{traceineq}
\tr({\Sigma}) = \sum\limits_{i=0}^{\ell} c_{p,\ell} ~ p^i(1-p)^{k-i} \sum\limits_{|\Omega| = i} \norm{{q_\Omega}}^2 \leq 2 \sum\limits_{i=0}^{\ell}  {{k}\choose{i}} p^i(1-p)^{k-i} \frac{i}{\zeta} \leq  \frac{2\,k\,p}{\zeta}.
\end{equation}
The first inequality follows from the implication that $kp \leq \ell $ $\Rightarrow$ $c_{p,\,\ell}\leq 2 $.  The second inequality is obtained by bounding the sum from above with the expectation of a binomial random variable with parameters $k$ and $p$.  Combining \eqref{widthtracerelation} and \eqref{traceineq} we have the desired result.

To complete the proof, we need to show that $\mathrm{rank}(\Sigma) = \dim(\eig)$ and that all the nonzero eigenvalues of $\Sigma$ are equal to each other.  For each $i \leq \ell$ denote $S^{(i)} := \sum\limits_{|\Omega| = i} {I_\Omega}' (({\proj_\eig})_{\Omega,\Omega})^{-1}{1_{|\Omega|}} {1_{|{\Omega}|}}'\allowbreak  (({\proj_\eig})_{\Omega,\Omega})^{-1}\allowbreak{I_\Omega}$ so that $\sum\limits_{|\Omega| = i} {q_{\Omega}(\eig)}{q_\Omega(\eig)}' = \proj_\eig S^{(i)} \proj_\eig$.  Let $\Pi \in \R^{k \times k}$ be a permutation matrix such that $\Pi A_\plant \Pi' = A_\plant$, i.e., $\Pi$ corresponds to an element of the automorphism group of $\plant$.  It is easily seen that $\Pi \proj_\eig \Pi = \proj_\eig$.  Consequently, if a vertex subset $\Omega \subset \{1,\dots,k\}$ is mapped to $\hat{\Omega}$ under the automorphism represented by $\Pi$, then we have that $({\proj_\eig})_{\Omega,\,\Omega} = ({\proj_\eig})_{\hat{\Omega},\hat{\Omega}}$.  In turn, one can check that $\Pi \, {I_\Omega}' (({\proj_\eig})_{\Omega,\Omega})^{-1}{1_{|\Omega|}} {1_{|{\Omega}|}}'  (({\proj_\eig})_{\Omega,\Omega})^{-1}{I_\Omega} \, \Pi' \allowbreak = {I_{\hat{\Omega}}}' (({\proj_\eig})_{\hat{\Omega},\hat{\Omega}})^{-1}{1_{|\hat{\Omega}|}} {1_{|{\hat{\Omega}}|}}'  (({\proj_\eig})_{\hat{\Omega},\hat{\Omega}})^{-1}{I_{\hat{\Omega}}}$.  Based on these observations and the fact that $|\Omega| = i \Leftrightarrow |\hat{\Omega}| = i$, we note that a summand of $S^{(i)}$ gets mapped to another summand of $S^{(i)}$ under conjugation by $\Pi$.  Moreover, automorphisms are injective functions, and hence distinct summands of $S^{(i)}$ must be mapped to distinct summands of $S^{(i)}$.  Thus, we conclude that $\Pi S^{(i)} \Pi' = S^{(i)}$ for each $i \leq \ell$ and for any permutation matrix $\Pi \in \R^{k \times k}$ representing an automorphism of $\plant$.

As $\plant$ is vertex- and edge-transitive, and as $\plant^c$ is also edge-transitive, each $S^{(i)}$ is of the following form:
\begin{align}\label{SumDecomposed}
(S^{(i)})_{p,q} &= \begin{cases} \alpha_1, \text{ if }({\Ap})_{p,q}=1 \text{ and } p\neq q  \\
\alpha_2, \text{ if }({\Ap})_{p,q} = 0 \text{ and } p\neq q\\
\alpha_3,\text{ if } p=q\end{cases}
\Longrightarrow  {S^{(i)}} = \alpha_1 A_\plant + \alpha_2 A_{\Gamma^c} + \alpha_3 I_k,
\end{align}
for some $\alpha_1,\alpha_2,\alpha_3 \in \R$.  Since $\Gamma $ is vertex-transitive it is also a regular graph, and consequently the discussion from Section \ref{subsubsec:regular} implies that the eigenspaces of $A_\plant$ and $A_{\plant^c}$ are the same.  As $\plant$ is assumed to be connected, we have from equations \eqref{regularcomplement},\eqref{SumDecomposed} and from the equality $\sum\limits_{|\Omega| = i} {q_{\Omega}(\eig)}{q_\Omega(\eig)}' = \proj_\eig S^{(i)} \proj_\eig$ that:
\begin{align*}
\sum\limits_{|\Omega| = i} {q_{\Omega}(\eig)}{q_\Omega(\eig)}' = \begin{cases} [\alpha_1\lambda_\eig+\alpha_2(k-\lambda_\eig-1)+\alpha_3] \frac{{1_k}{1_k}^T}{k}, \text{ if } \eig = \mathrm{span}\{1_k 1_k'\}, \\   [\alpha_1\lambda_\eig-\alpha_2(\lambda_\eig+1)+\alpha_3] \proj_\eig, \text{ otherwise.}\end{cases}
\end{align*}
Therefore, $\sum\limits_{|\Omega| = i} {q_{\Omega}(\eig)}{q_\Omega(\eig)}'$ has rank equal to $\dim(\eig)$ and its nonzero eigenvalues are equal to each other.  Since this holds for each $i \leq \ell$, we conclude from \eqref{SigmaFirstEq} that $\mathrm{rank}(\Sigma) = \mathrm{dim}(\eig)$ and that all the eigenvalues of $\Sigma$ are equal to each other.  This completes the proof.
\end{proof}

\subsubsection{Simplifications based on Coherence}  \label{Subsection MutualCoherence}
The Kruskal rank of a subspace is intractable to compute in general; as a result, a number of subspace parameters have been considered in the literature to obtain tractable bounds on the Kruskal rank.  The most prominent of these is the coherence parameter of a subspace.  In our context, the additional analytical simplification provided by the coherence of a subspace along with Proposition \ref{Prop CombWidth Upperbound} lead to simple performance guarantees on the Schur-Horn relaxation for symmetric planted graphs.

\begin{definition}
Let $\mathcal{S} \subseteq \R^k$ be a subspace.  The \emph{coherence} of $\mathcal{S}$, denoted $\mu(\mathcal{S})$, is defined as:
\begin{align*}
\mu(\mathcal{S}) := \max\limits_{\substack{1\leq i,j\leq k \\ i\neq j}}  \frac{|{(\proj_{\mathcal{S}})}_{i,j}|}{({(\proj_{\mathcal{S}})}_{i,i})^{1/2} ({(\proj_{\mathcal{S}})}_{j,j})^{1/2}}.
\end{align*}
\end{definition}
The coherence parameter of a subspace can be computed efficiently, and it can be used to bound the Kruskal rank from below:
\begin{proposition} \label{Prob krank bound}\cite{donoho2003optimally}
For any subspace $\mathcal{S} \in \R^k$, $\krank{\mathcal{S}} \geq \frac{1}{\mu(\mathcal{S})}$.
\end{proposition}

Further, for symmetric planted graphs $\plant$, the following result provides a bound on the minimum eigenvalue of minors of $\proj_\eig$ for eigenspaces $\eig$ of $A_\plant$.  Recall that this result is directly relevant in the context of Proposition \ref{Prop CombWidth Upperbound}.

\begin{proposition}\label{Prop lambda0 bound}
Suppose $\plant$ is a vertex-transitive graph with adjacency matrix $A_\plant \in \Sym^k$, and let $\eig$ denote an eigenspace of $A_\plant$.  For any $\ell \in \mathbb{Z}$ with $\ell < \frac{1}{\mu(\eig)} + 1$, we have that $\min\limits_{\substack{\Omega \subset \{1,\dots,k\}  \\ |\Omega| \leq \ell}} \lambda_{\mathrm{min}} \big((\proj_{\eig})_{\Omega,\,\Omega}\big)\geq \frac{\dim(\eig)}{k} \,\big(1-(\ell-1)\mu(\eig)\big)$.
\end{proposition}

\begin{proof}
One can check that $\Pi \proj_\eig \Pi' = \proj_\eig$ for permutation matrices $\Pi \in \R^{k \times k}$ that correspond to automorphisms of $\plant$.  Therefore, by vertex transitivity, the diagonal entries of $\proj_\eig$ are all equal to each other.  As $\tr(\proj_\eig) = \dim(\eig)$, we conclude that $({\proj_\eig})_{i,i} = \frac{\dim(\mathcal{E})}{k}$ for each $i=1,\dots,k$.  Every row of ${(\proj_\eig)_{\Omega,\,\Omega}}$ has at most $\ell-1$ off-diagonal entries, and each of these entries is bounded above by $\frac{\dim(\eig)}{k}\mu(\eig)$.  We obtain the desired result by applying the Gershgorin circle theorem.
\end{proof}

\subsection{Main Result} \label{Subsection MainResult}
Building on the preceding discussion, we state and prove our main result Theorem (\ref{Main Theorem}).  The proof of this result relies on an intermediate step regarding the $M_{22}$ submatrix of the dual variable $M = \begin{pmatrix} M_{11} & M_{12} \\ M_{12}' & M_{22}\end{pmatrix}$ from Proposition \ref{Optimality Conditions}.  From that result, we are required to obtain an $M_{22} \in \Sym^{n-k}$ such that $(i)$ For each $i,j = 1,\dots,n-k$ we have $(M_{12})_{i,j} = 1$ if $(A_\obs)_{i+k,j+k} = 1$ or if $i=j$, and $(ii)$ The operator norm $\|M_{22}\|_2$ is as small as possible.

We present the following result from \cite{BandeiraSpectralNormIndependentEntries}, which we utilize subsequently in Lemma \ref{M22 lemma} to establish a bound on $\|M_{22}\|_2$:

\begin{lemma} \cite{BandeiraSpectralNormIndependentEntries} \label{Afonso Lemma}
Let $X \in \Sym^d$ be a symmetric matrix whose entries $X_{i,j}$ are independent and centered random variables.  For each $\epsilon \in (0,1/2]$, there exists a constant $\tilde{c}_\epsilon$ such that for all $x \geq 0$:
\begin{align*}
\mathbb{P} (\|X\|_2 \geq (1+\epsilon)2\tilde{\sigma} + x ) \leq d \exp{\left(-\frac{x^2}{\tilde{c}_\epsilon \tilde{\sigma}_*^2}\right)},
\end{align*}
where $\tilde{\sigma} := \max\limits_i \sqrt{\sum\limits_j \mathbb{E}[X_{i,j}^2]}$ and each $|X_{i,j}|\leq \tilde{\sigma}_*$ almost surely.
\end{lemma}

\begin{lemma} \label{M22 lemma}
Consider a planted subgraph problem instance in which the nodes of $\obs$ and $\plant$ are labeled so that the leading principal minor of $A_\obs$ of order $k$ is equal to $A_\plant$, and the remaining edges in $\obs$ are drawn according to the Erd\H{o}s-R\'{e}nyi model with probability $p \in [0,1)$. For constants $c_1 = \sqrt{\frac{9p}{1-p}}$ and $c_2$ depending only on $p$ and for $\alpha \geq c_1 \sqrt{n-k}$, there exists $M_{22} \in \Sym^{n-k}$ satisfying
\begin{enumerate}
\item $(M_{22})_{i,j} = 1 \text{ if } (\Ag)_{i+k,j+k}=1 \text{ or } i=j$,
\item $\|M_{22}\|_2 < \alpha$,
\end{enumerate}
with probability at least $1- (n-k) \exp\Big( -c_2\, \big(\alpha-c_1\sqrt{n-k}\big)^2 \Big)$.
\end{lemma}
\begin{proof}
Our proof is inspired by the approach in \cite{ames2011nuclear}. Consider the following matrix $M_{22} \in \Sym^{n-k}$:
\begin{align}
(M_{22})_{i,j} = \begin{cases} 1, \text{ if } (\Ag)_{i+k,j+k}=1, i\neq j \\
						 	 \frac{-p}{1-p}, \text{ if } (\Ag)_{i+k,j+k}=0, i\neq j \\
						      0, \text{ if } i = j.
				\end{cases}
\end{align}
As the submatrix $(A_\obs)_{i+k,j+k}, ~i,j=1,\dots,n-k$ consists of independent and centered entries (in the off-diagonal locations) and zeros on the diagonal, one can check that $M_{22}$ is a random matrix that satisfies the requirements of Lemma $\ref{Afonso Lemma}$.  Further, the first part of the present lemma is satisfied.  The second claim follows from an application of Lemma $\ref{Afonso Lemma}$ with $\epsilon = 1/2$.
\end{proof}

Combining Proposition \ref{Optimality Conditions}, Proposition \ref{M12 Proposition}, and Lemma \ref{M22 lemma}, we now state and prove the main result of this paper:
\begin{theorem} \label{Main Theorem}
Consider a planted subgraph problem instance in which the nodes of $\obs$ and $\plant$ are labeled so that the leading principal minor of $A_\obs$ of order $k$ is equal to $A_\plant$, and the remaining edges in $\obs$ are drawn according to the Erd\H{o}s-R\'{e}nyi model.  Suppose $\eig \subset \R^k$ is an eigenspace of $A_\plant$ with associated eigenvalue $\lambda_\eig$, and we employ the Schur-Horn relaxation \eqref{PrimalOptimization} with parameter $\gamma = \lambda_\eig$. Further suppose that:
\begin{enumerate}
\item $p \in [0,\frac{\krank{\eig}}{k})$,
\end{enumerate}
and for some $\ell \in \mathbb{Z}$ satisfying $kp < \ell  \leq \krank{\eig} $,
\begin{enumerate}
\setcounter{enumi}{1}
\item $n < \min\big(  \frac{\egap(A_\plant,\eig)^2}{4  \omega(\eig,\ell,p)} , \frac{( \egap(A_\plant,\eig) - 2|\lambda_\eig| )^2}{4 c_1^2}  \big) +k $.
\end{enumerate}
Then the Schur-Horn relaxation succeeds at identifying the planted subgraph $\plant$ inside $\obs$ with probability at least $1- p_1 - p_2 $, where: \newline $p_1 = \allowbreak 1 - \Bigg[\big(1- \exp{(-\tfrac{(\ell-kp)^2}{\ell + kp})}\big)^{n-k} \times \allowbreak \Big(1-2k \exp{\Big(-\tfrac{3 \zeta  \big(\tfrac{1}{4}\mathrm{eigengap}(A_\plant, \eig )^2 -(n-k) \omega(\eig,\ell,p)\big)^2}  {4 \ell\big(\tfrac{1}{4}{\mathrm{eigengap}(A_\plant, \eig )}^2 + 2(n-k) \omega(\eig,\ell,p)\big)}\Big) }\Big) \Bigg]$ and $p_2 = (n-k) \times \allowbreak \exp\Big( -c_2 \allowbreak \big( \tfrac{1}{2}\mathrm{eigengap(A_\Gamma,\eig)}-|\lambda_\eig|-c_1\sqrt{n-k} \big)^2 \Big)$.
Here $\zeta = \min\limits_{\substack{\Omega \subset \{1,\dots,k\}  \\ |\Omega| \leq \ell}} \lambda_{\mathrm{min}} \big((\proj_{\eig})_{\Omega,\Omega}\big)$, and the constants $c_1 = \sqrt{\frac{9p}{1-p}}$ and $c_2$ depend only on $p$.
\end{theorem}

\begin{proof}
As discussed previously, since $\ell \leq \krank{\eig}$ we have that $\zeta > 0$.  We establish the result by constructing a dual certificate $M$ satisfying the conditions of Proposition $\ref{Optimality Conditions}$.

We start by setting $M_{11} = A_\plant$.  This ensures that conditions $(\ref{Optimality Condition M11Comonotone})$ and $(\ref{Optimality Condition M11 restricted})$ of Proposition $\ref{Optimality Conditions}$ are immediately satisfied.  Next, we choose $M_{12}$ as discussed in Proposition $\ref{M12 Proposition}$, with the parameter $\delta = \frac{1}{2}\,\mathrm{eigengap}(\plant,\eig)$, which satisfies $\delta \geq \sqrt{ (n-k) \omega(\eig,\ell,p)}$ due to the upper bound on $n$. Such an $M_{12}$ exists with probability at least $1-p_1$, and it satisfies condition $(\ref{Optimality Condition M12Columns})$ of Proposition $\ref{Optimality Conditions}$ as well as the bound $\|M_{12}\|_2 < \frac{1}{2} \mathrm{eigengap}(A_\plant,\eig)$. Finally, we set $M_{22}$ as discussed in Lemma $\ref{M22 lemma}$, with $\alpha =\frac{1}{2} \mathrm{eigengap}(A_\plant,\eig) - |\lambda_\eig|$, which satisfies $\alpha \geq c_1 \sqrt{n-k}$ due to the upper bound on $n$.  Such an $M_{22}$ exists with probability at least $1-p_2$ and satisfies the bound $\|M_{22}\|_2 < \frac{1}{2}\mathrm{eigengap}(A_\plant,\eig) -|\lambda_\eig|$.

Based on this construction, the matrix $M=\begin{pmatrix}
M_{11}&M_{12}\\M_{12}'&M_{22}
\end{pmatrix}$ satisfies conditions $(\ref{Optimality Condition E constraint})$ and $(\ref{Optimality Condition Eigengap})$ of Proposition $\ref{Optimality Conditions}$.  Thus, if $M_{12}$ and $M_{22}$ with the stated properties exist, then all the conditions of Proposition $\ref{Optimality Conditions}$ are satisfied.  By the union bound, the desired $M_{12}$ and $M_{22}$ exist concurrently with probability at least $1-p_1-p_2$.
\end{proof}



\begin{remark} \label{remark l}
The parameter $\ell$ arises in multiple aspects of this result.  We discuss specific choices of $\ell$ in the corollaries in the next section.
\end{remark}

Theorem \ref{Main Theorem} provides a non-asymptotic bound on the performance of the Schur-Horn relaxation (\ref{PrimalOptimization}).  In words, this relaxation succeeds with high probability in identifying a subgraph $\plant$ planted inside a larger graph $\obs$ (under the Erd\H{o}s-R\'{e}nyi model) provided $A_\plant$ has an eigenspace $\eig$ satisfying four conditions: $(i)$ The eigenspace $\eig$ has large Kruskal rank, $(ii)$ The eigenspace $\eig$ has small combinatorial width, $(iii)$ $A_\plant$ has a large eigengap with respect to $\eig$, and $(iv)$ The projection matrix $\proj_\eig$ has the property that all sufficiently large principal minors are well-conditioned.  In practice, larger dimensional eigenspaces of $A_\plant$ may be expected to satisfy these conditions more easily, and therefore we set $\gamma$ equal to the eigenvalue of $A_\plant$ of largest multiplicity in our experimental demonstrations in Section \ref{SectionNumerical}.




\subsection{Specializations of Theorem \ref{Main Theorem}} \label{Subsection Simplifications}
We appeal to the discussion in Section \ref{properties} on the properties of the Kruskal rank and the combinatorial width to obtain specializations of Theorem \ref{Main Theorem} to certain graph families.  We begin by considering the case of symmetric planted graphs with symmetric complements:

\begin{corollary} \label{Main Corollary}
Consider a planted subgraph problem instance in which the nodes of $\obs$ and $\plant$ are labeled so that the leading principal minor of $A_\obs$ of order $k$ is equal to $A_\plant$, and the remaining edges in $\obs$ are drawn according to the Erd\H{o}s-R\'{e}nyi model.  Suppose $\eig \subset \R^k$ is an eigenspace of $A_\plant$ with associated eigenvalue $\lambda_\eig$, and we employ the Schur-Horn relaxation \eqref{PrimalOptimization} with parameter $\gamma = \lambda_\eig$.  Further suppose that the following three conditions hold:
\begin{enumerate}
\item $\Gamma$ is a connected symmetric graph with a symmetric complement, \label{Corollary req Symmetric}
\item $p \in [0,\frac{1}{\mu(\eig) k})$, \label{Corollary req p interval}
\item $n < \min\Big(\frac{\mathrm{eigengap}(A_\plant,\eig)^2\, \dim(\eig)^2 \,  \big(1-kp\mu(\eig)\big)}{16k^2p}, \frac{(\egap(A_\plant,\eig) - 2|\lambda_\eig|  )^2}{4 c_1^2}\Big) +k.$ \label{Corollary requirement n bound}
\end{enumerate}
Then the Schur-Horn relaxation succeeds in identifying the planted subgraph $\plant$ inside the larger graph $\obs$ with probability at least $1-p_1 - p_2 $, where $p_1$ and $p_2$ are as stated in Theorem \ref{Main Theorem} (one can substitute $\frac{4 k^2 p}{\dim(\eig)^2 \big(1-kp\mu(\eig)\big)}$ for the $\omega(\eig,\ell,p)$ term appearing in $p_1$).
\end{corollary}

\begin{proof}
This result follows by a combination of Theorem \ref{Main Theorem},  and Propositions $\ref{Prop CombWidth Upperbound}$, $\ref{Prob krank bound}$, $\ref{Prop lambda0 bound}$.  Set $\ell = \lceil\frac{1}{2}(kp +\frac{1}{\mu(\eig)}) \rceil$.  This choice satisfies $ kp < \ell \leq  \krank{\eig}$ based on Proposition \ref{Prob krank bound}.  One can also check that the inequality $\ell < \frac{1}{\mu(\eig)} + 1$ holds.  The vertex transitivity of $\plant$ implies that one can appeal to Proposition \ref{Prop lambda0 bound} to conclude that
\begin{equation} \label{lowerboundonzeta}
\min\limits_{\substack{\Omega \subset \{1,\dots,k\}  \\ |\Omega| \leq \ell}} \lambda_{\mathrm{min}} \big((\proj_{\eig})_{\Omega,\,\Omega}\big)\geq \frac{\dim(\eig)}{k} \big(1-(\ell -1)\mu(\eig)\big) > \frac{\dim(\eig)}{2k} (1-kp\mu(\eig)).
\end{equation}
Based on the condition on $p$, this lower bound is strictly positive.  As $\plant$ is symmetric and has a symmetric complement (and is connected), we conclude from Proposition \ref{Prop CombWidth Upperbound} that $\omega(\eig,\ell,p) \leq \frac{4 k^2 p}{\dim(\eig)^2 (1-kp\mu(\eig))}$.

Finally, one can check that conditions (\ref{Corollary req p interval}) and (\ref{Corollary requirement n bound}) of the corollary imply that both of the requirements of Theorem \ref{Main Theorem} are met, and hence the Schur-Horn relaxation succeeds in identifying the planted subgraph $\plant$ with probability at least $1-p_1-p_2$, where $p_1$ and $p_2$ are as stated in Theorem \ref{Main Theorem} -- one can substitute $\frac{4 k^2 p}{\dim(\eig)^2 \big(1-kp\mu(\eig)\big)}$ as an upper bound for $\omega(\eig,\ell,p)$ and $\frac{\dim(\eig)}{2k} (1-kp\mu(\eig))$ as a lower bound for $\zeta$, which yields a lower bound on $1-p_1-p_2$ from Theorem \ref{Main Theorem}.
\end{proof}

As the coherence parameter of an eigenspace is more tractable to compute than the Kruskal rank, this result provides an efficiently verifiable set of conditions that guarantee the success of the Schur-Horn relaxation \eqref{PrimalOptimization} for symmetric planted graphs $\plant$.  This result specialized to the case of the planted clique problem yields the result of Ames and Vavasis \cite{ames2011nuclear}.

\begin{corollary} \label{Corollary Clique}
Fix $p \in [0,1)$ and consider a family of planted clique problem instances $\{\plant_k, \obs_k\}_{k=1}^\infty$ generated according to the Erd\H{o}s-R\'{e}nyi model, where $\plant_k$ is the $k$-clique and $\obs_k$ is a graph on $n_k$ nodes.   There exists a constant $\beta > 0$ only depending on $p$ such that if $n_k \leq \beta k^2$, the Schur-Horn relaxation with $\gamma = -1$ succeeds in identifying $\plant_k$ inside $\obs_k$ with probability approaching one exponentially fast in $k$.
\end{corollary}

\begin{proof}
The $k$-clique is a connected symmetric graph with a complement that is also symmetric; hence the first condition of Corollary \ref{Main Corollary} is satisfied.  Each $A_{\plant_k} \in \Sym^k$ has a $(k-1)$-dimensional eigenspace $\eig$ such that $\mu(\eig) = \frac{1}{k-1}$, $\dim(\eig) = k-1$, and $\egap(A_\plant,\eig) = k$.

Based on the choice $\ell = \lceil \frac{1}{2}(kp + \frac{1}{\mu(\eig)}) \rceil$ as in Corollary (\ref{Main Corollary}), one can check that $\frac{\dim(\eig)}{2k}\big( 1-kp\mu(\eig)\big)  = \Theta(1)$, that $\omega(\eig,\ell,p) \leq \frac{4 k^2 p}{\dim(\eig)^2\,\big( 1-kp\mu(\eig)\big)}  = \Theta(1)$, and that $\ell-kp =  \Theta(k)$.

Set $n_k = \frac{k^2}{32}\min(  \frac{1-kp\mu(\eig)}{2p}, \frac{1}{{c_1}^2}) + k$. One can check that the third condition of Corollary \ref{Main Corollary} is satisfied with this choice.  Moreover, this value of $n_k$ (or any smaller value) yields $\frac{1}{4}\egap(\Ap,\eig)^2 - (n_k-k)\omega(\eig,\ell,p) = \Theta(k^2)$ and $\frac{1}{2}\egap(\Ap,\eig)-1-c_1\sqrt{n_k-k} = \Theta(k)$.

By Corollary \ref{Main Corollary}, we conclude that the Schur-Horn relaxation (\ref{PrimalOptimization}) with parameter $\gamma=-1$ identifies a hidden $k$-clique with probability $1-p_1-p_2$, where
\begin{align*}
p_1 &= 1- \big(1-\exp(-c_4 k)\big)^{n-k} \big(1-2k\exp(-c_3 k)\big) \longrightarrow 0,\text{ as } n,~k \to \infty, \text{ and}\\
p_2 &= (n-k) \exp(-c_5 k^2) \longrightarrow 0,\text{ as } n,~k \to \infty,
\end{align*}
for some constants $c_3>0$, $c_4>0$, and $c_5>0$.
\end{proof}

Thus, although our main result Theorem \ref{Main Theorem} was developed in a non-asymptotic setting, it can be specialized to obtain the asymptotic result presented in \cite{ames2011nuclear}.



\section{Numerical Experiments}\label{SectionNumerical}

%

\subsection{Semidefinite Descriptions of the Schur-Horn Orbitope} \label{subsec:sdpschurhorn}


We begin with a discussion of semidefinite representations of the Schur-Horn orbitope $\sh(M)$ for $M \in \mathbb{S}^n$.  Specifically, suppose $s_\ell: \Sym^n \rightarrow \R$ denotes the sum of the $\ell$-largest eigenvalues of a symmetric matrix for $\ell=1,\dots,n$.  Then the Schur-Horn orbitope $\sh(M)$ can be described via majorization inequalities on the spectrum \cite{sanyal2011orbitopes}:
\begin{equation} \label{SH majorization}
\sh(M) = \left\{N \in \Sym^n ~|~ s_\ell(N) \leq s_\ell(M) ~\mathrm{for}~ 1 \leq \ell \leq n-1, ~\mathrm{and}~ \tr(N) = \tr(M) \right\}.
\end{equation}
As the sublevel sets of the convex functions $s_\ell$ have tractable semidefinite descriptions \cite{ben2001lectures}, one can obtain a lifted polynomial-sized semidefinite representation of $\sh(M)$ for arbitrary $M \in \Sym^n$.  However, specifications of $\sh(M)$ via semidefinite representations of the sublevels sets of $s_\ell$ involve a total of $O(n)$ additional matrix variables in $\mathbb{S}^n$ and $O(n)$ semidefinite constraints (one for each of the majorization inequalities in \eqref{SH majorization}); in particular, these do not take advantage of any structure in the spectrum of $M$, such as multiplicities in the eigenvalues.

We discuss next an alternative semidefinite representation of $\sh(M)$ that is based on a modification of the description of $\sh(M)$ presented in \cite{ding2009low}, and it exploits the multiplicities in the eigenvalues of $M$ so that both the number of additional matrix variables and semidefinite constraints scale with the number of distinct eigenvalues of $M$ rather than the ambient size $n$ of $M$.  Suppose $M$ has $q$ distinct eigenvalues $\lambda_1,\dots,\lambda_q$ with multiplicities $m_1,\dots,m_q$. Then one can check that \cite{ding2009low}:
\begin{equation} \label{SHDescription2}
\begin{aligned}
\sh(M) = \Big\{N \in \mathbb{S}^n ~|~ & \exists Y_i \in \Sym^n, Y_i \succeq 0, ~ i = 1,\dots,q ~\mathrm{such~that} \\ & N = \sum\limits_{i=1}^{q} \lambda_i Y_i,~ \sum\limits_{i=1}^{q} Y_i = I_n, ~\tr(Y_i) = m_i ~\mathrm{for}~ i=1,\dots,q\Big\}.
\end{aligned}\end{equation}
In this latter description of the Schur-Horn orbitope, both the number of additional matrix variables in $\mathbb{S}^n$ and the number of semidefinite constraints are on the order of the number of distinct eigenvalues of $M$, which can be far smaller than $n$ for the adjacency matrices of graphs considered in this paper. In the numerical experiments presented next, we employ the description \eqref{SHDescription2} of the Schur-Horn orbitope.


\subsection{Experimental Results} \label{expresults}
\begin{figure*}
\centering
\small
\begin{tabular}{|c|c|c|}
\hline
{\bf Planted graph $\plant$} & {\bf Eigenvalues} & {\bf Kruskal rank of} \\
{\bf [with $\#$ vertices]} & {\bf [with multiplicity]} & {\bf largest eigenspace} \\
\hline \hline
  Clebsch [$k = 16$] & {\multirow{2}{*}{$5 [\times 1], -3 [\times 5], 1 [\times 10]$}} & {\multirow{2}{*}{$5$}} \\
  (Figure \ref{ClebschGraph}) & & \\
  \hline
  Generalized & {\multirow{4}{*}{$10 [\times 1], -5 [\times 6], 1 [\times 20]$}} & {\multirow{4}{*}{$8$}} \\
  quadrangular-$(2,4)$ & & \\
  $[k=27]$ & & \\
  (Figure \ref{fig:GQ24}) & & \\
  \hline
  $8$-Triangular $[k=28]$ & {\multirow{2}{*}{$12 [\times 1], 4 [\times 7], -2 [\times 20]$}} & {\multirow{2}{*}{$6$}} \\
  (Figure \ref{fig:Triangular8}) & & \\
  \hline
  $9$-Triangular $[k=36]$ & {\multirow{2}{*}{$14 [\times 1], 5 [\times 8], -2 [\times 27]$}} & {\multirow{2}{*}{$7$}} \\
  (Figure \ref{fig:Triangular9}) & & \\
  \hline
\end{tabular}
\caption{Planted subgraphs (and associated parameters) for which we demonstrate the utility of the Schur-Horn relaxation.  See Figure \ref{PhaseTransitions} for the associated phase transitions.} \label{experiments}
\end{figure*}

%
%
%
%
%
%

We investigate the performance of the Schur-Horn relaxation (\ref{PrimalOptimization}) in planted subgraph problems with the four planted subgraphs $\plant$ listed in Figure \ref{experiments}.  For each of these graphs, we set $\gamma$ equal to the eigenvalue corresponding to the largest eigenspace of the corresponding graph.  We vary $n$ (the size of the larger graph $\obs$ inside which $\plant$ is planted) and $p$ (the probability of a noise edge in $\obs$), and we obtain $10$ random instances of planted subgraph problems for each value of $n$ and $p$.  In Figure \ref{PhaseTransitions}, we plot the empirical probability of success of the Schur-Horn relaxation for these random trials; the white cells represent a probability of success of one and the black cells represent a probability of success of zero.  Our results were obtained using the CVX parser \cite{grant2008cvx,grant2008graph} and the SDPT3 solver \cite{toh1999sdpt3}.  In each of the four cases, the Schur-Horn relaxation (\ref{PrimalOptimization}) succeeds in solving the underlying planted subgraph problem for suitably small $n$ and $p$.

\begin{figure}[hbt]
\centering
   \subcaptionbox{\label{fig:ClebschPhase}}{\includegraphics[width=.40\linewidth, height=.3\linewidth]{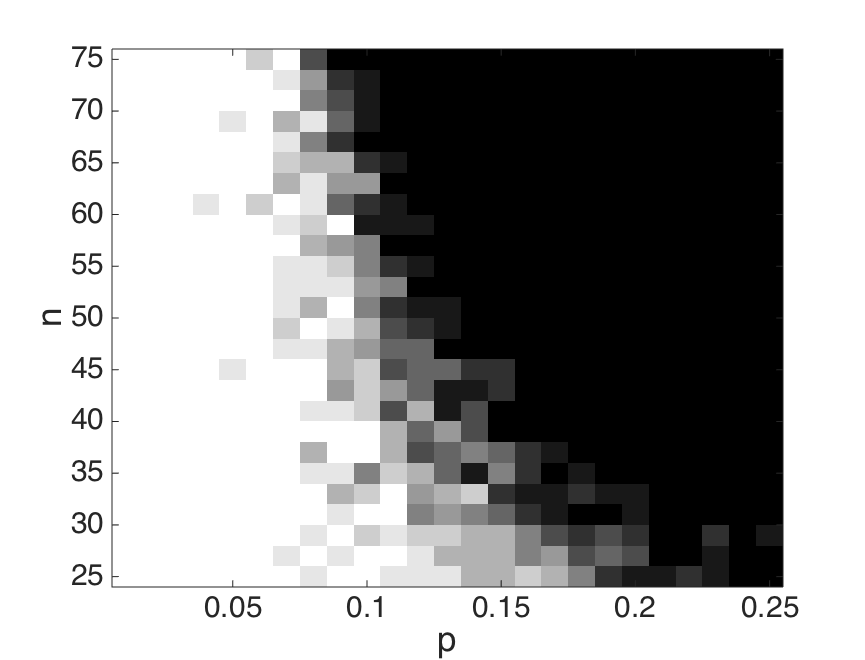}}
   \subcaptionbox{\label{fig:GQ24Phase}}{\includegraphics[width=.40\linewidth, height=.3\linewidth]{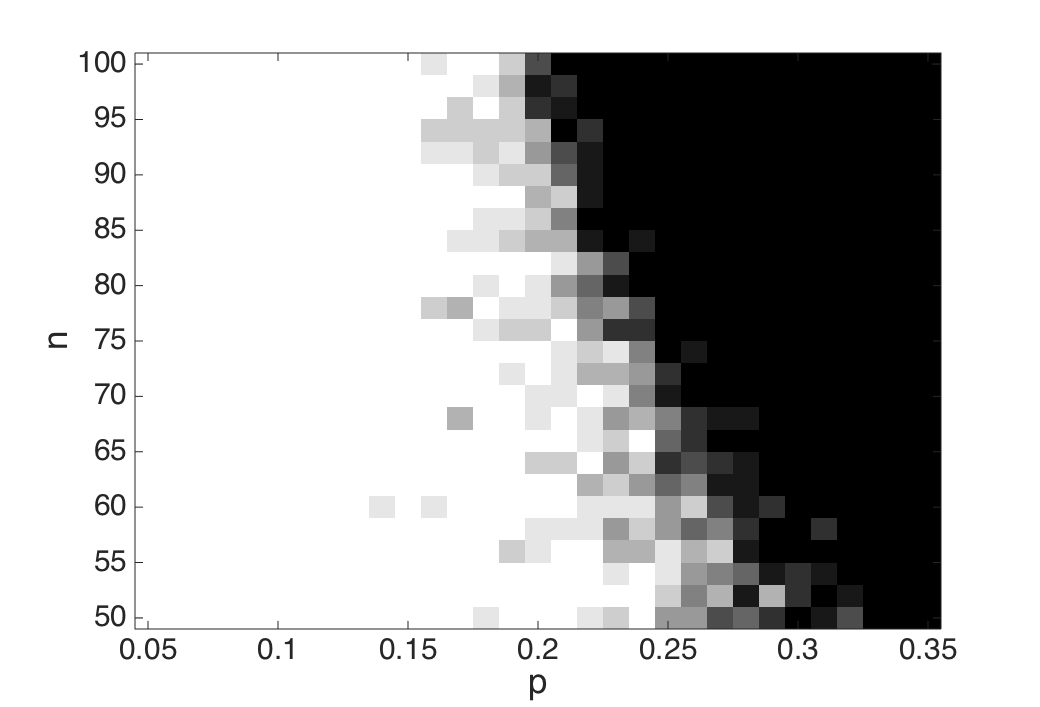}}  \\
   \subcaptionbox{\label{fig:Triangular8Phase}}{\includegraphics[width=.40\linewidth, height=.3\linewidth]{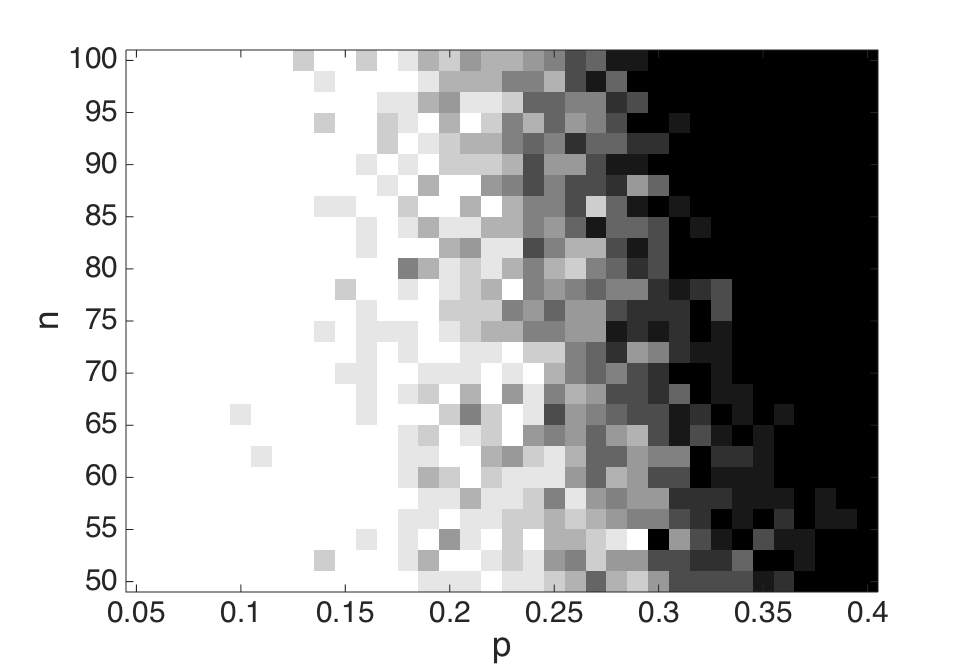}}
  \subcaptionbox{\label{fig:Triangular9Phase}}{\includegraphics[width=.40\linewidth, height=.3\linewidth]{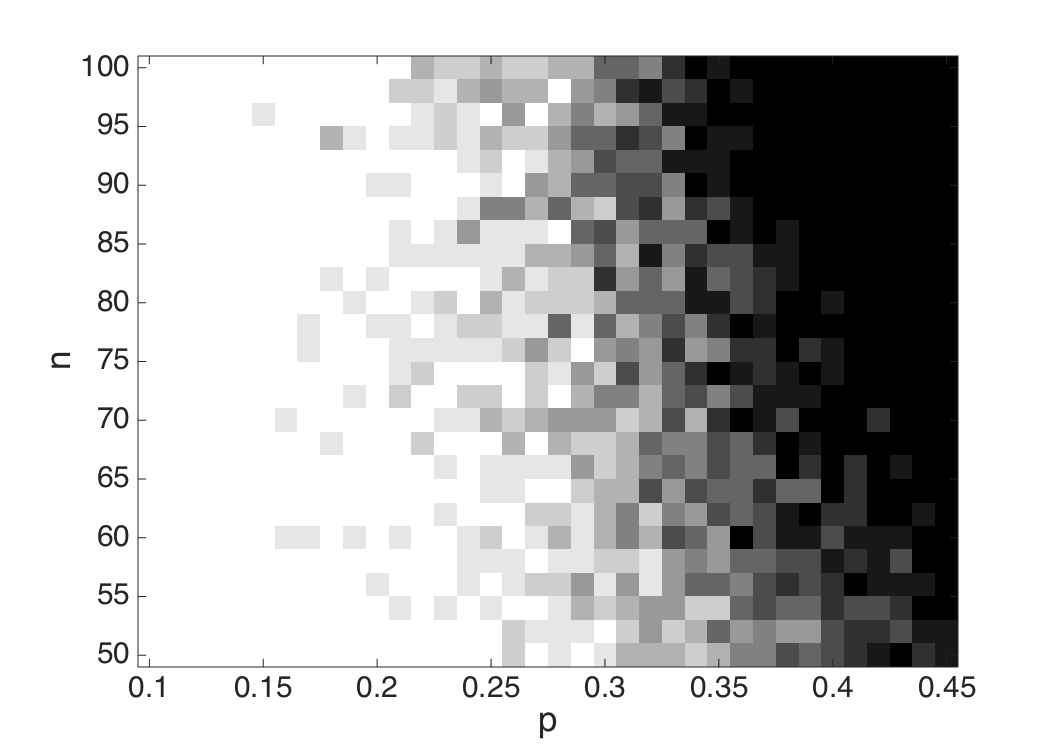}}
  \caption{Phase transition plots based on the experiment described in Section \ref{expresults} for the (\subref{fig:ClebschPhase}) Clebsch graph, (\subref{fig:GQ24Phase}) Generalized quadrangle-$(2,4)$ graph, (\subref{fig:Triangular8Phase}) $8$-Triangular graph,  and (\subref{fig:Triangular9Phase}) $9$-Triangular graph.} \label{PhaseTransitions}
\end{figure}

\section{Discussion} \label{SectionDiscussion}

In this paper, we introduce a new convex relaxation approach for the planted subgraph problem, and we describe families of problem instances for which our method succeeds.  Our method generalizes previous convex optimization techniques for identifying planted cliques based on nuclear norm minimization \cite{ames2011nuclear}, and it is useful for identifying plangted subgraphs consisting of few distinct eigenvalues.  There are several further directions that arise from our investigation, and we mention a few of these here.

\paragraph{Spectrally comonotone matrices with sparsity constraints} One of the ingredients in the proof of our main result Theorem \ref{Main Theorem} is to find a matrix $M_{11} \in \Sym^k$ that is spectrally comonotone with $A_\plant \in \Sym^k$, and that further satisfies the condition that $(M_{11})_{i,j} = 1$ whenever $(A_\plant)_{i,j} = 1$.  In the proof of Theorem \ref{Main Theorem} we simply choose $M_{11} = A_\plant$.  This choice does not exploit the fact that the entries of $M_{11}$ corresponding to those where $(A_\plant)_{i,j} = 0$ are not constrained (and in particular can be nonzero).  With a different choice of $M_{11}$, one could replace $\egap(A_\plant,\eig)$ in Theorem \ref{Main Theorem} by $\egap(M_{11},\eig)$ (recall that $\eig$ is an eigenspace of $A_\plant$).  Consequently, our main result could be improved via principled constructions of matrices $M_{11} \in \Sym^k$ that satisfy the conditions of Theorem \ref{Main Theorem} and for which $\egap(M_{11},\eig) > \egap(A_\plant,\eig)$.


\paragraph{Sparse graphs with eigenspaces with large Kruskal rank} One of the central questions concerning the planted subgraph problem is the possibility of identifying `sparse' planted subgraphs inside `dense' noise via computationally tractable approaches.  Concretely, suppose $\plant$ is a regular graph with degree $d$. Under the Erd\H{o}s-R\'{e}nyi model for the noise, the average degree of any $k$-node subgraph of the larger graph $\obs$ is about $(k-1)p$.  From Theorem \ref{Main Theorem}, we have that the Schur-Horn relaxation succeeds (with high probability) in identifying $\plant$ if $p \in [0, \tfrac{\krank{\eig}}{k})$, where $\eig \subset \R^k$ is one of the eigenspaces of $\plant$.  In other words, (for suitably large $k$) if $d < \krank{\eig}$ then the Schur-Horn relaxation succeeds in identifying $\plant$ inside $\obs$ despite the fact that $\plant$ is sparser than a typical $k$-node subgraph in $\obs$.  Of the graphs we have investigated in this paper, the Clebsch graph from Figure \ref{ClebschGraph} is an example in which both the degree and the Kruskal rank of the largest subspace are equal to $5$.  For some of the other small graphs discussed in this paper, the degree is larger than the Kruskal ranks of the eigenspaces.  For larger graphs, the computation of the Kruskal rank of the large eigenspaces quickly becomes computationally intractable.  Therefore, it is of interest to identify graph families in which (by construction) the degree is smaller than the Kruskal rank of one of the eigenspaces.

\paragraph{Convex geometry and graph theory} In developing convex relaxations for the planted subgraph problem (based on the formulation \eqref{Unsolvable Opt}) as well as other inverse problems involving unlabeled graphs, the key challenge is one of obtaining tractable convex outer approximations of the set $\mathcal{A}(B) = \{\Pi B \Pi' ~|~ \Pi ~\mathrm{is~an}~ n \times n ~\mathrm{permutation~matrix} \}$ for some given adjacency matrix $B \in \Sym^n$.  In particular, a convex approximation $\mathcal{C}$ that contains $\mathcal{A}(B)$ is useful if the normal cone $\mathcal{N}_{\mathcal{C}}(B)$ is large; as an example, the Schur-Horn relaxation $\sh(B)$ has this property for adjacency matrices $B$ with few distinct eigenvalues.  More generally, what is an appropriate convex relaxation for other structured graph families such as low-treewidth graphs (arising in inference in statistical graphical models), or graphs with a specified degree distribution (arising in social network analysis)?  Recent work \cite{chandrasekaran2012convex} provides a catalog of \emph{convex graph invariants} that are useful for obtaining computationally tractable convex relaxations of $\mathcal{A}(B)$.  A deeper investigation of the interaction between convex-geometric aspects of these invariants (such as the normal cones of the associated convex relaxations) and the structural properties of the graph specified by the adjacency matrix $B$ has the potential to yield new convex relaxations for general inverse problems on graphs.

%

\bibliographystyle{plain}
\bibliography{SchurHornBib}

\end{document}